\documentclass[11pt,twoside,english,final]{amsart}
\usepackage{amsmath}
\usepackage{amsthm}
\usepackage{amssymb}
\usepackage{bm}

\makeatletter
\theoremstyle{plain}
\newtheorem{thm}{\protect\theoremname}[section]
\theoremstyle{definition}
\newtheorem{defn}[thm]{\protect\definitionname}

\theoremstyle{plain}
\newtheorem{lem}[thm]{\protect\lemmaname}
\theoremstyle{plain}
\newtheorem{prop}[thm]{\protect\propositionname}
\theoremstyle{plain}
\newtheorem{cor}[thm]{\protect\corollaryname}

\theoremstyle{remark}
\newtheorem{rem}[thm]{Remark}

\newtheorem*{acknowledgement*}{Acknowledgement}

\usepackage{tikz-cd}

	 \renewcommand{\theenumi}{\emph{\roman{enumi})}} 
	\numberwithin{equation}{section}

\newcommand{\coloneqq}{\mathrel{\mathop:}=}

\newcommand{\hooklongrightarrow}{\mathbin{\lhook\joinrel\longrightarrow}}

\DeclareMathOperator{\Irr}{Irr}

\DeclareMathOperator{\Ind}{Ind}
\DeclareMathOperator{\Res}{Res}
\DeclareMathOperator{\res}{res}

\DeclareMathOperator{\Stab}{Stab}

\DeclareMathOperator{\tr}{tr}

\DeclareMathOperator{\Ker}{Ker}

\DeclareMathOperator{\ord}{ord}

\DeclareMathOperator{\GL}{GL}

\DeclareMathOperator{\M}{M}

\newcommand\leftexp[2]{\hspace{1pt}{\vphantom{#2}}^{#1}\hspace{-1pt}{#2}}

\newcommand{\C}{\ensuremath{\mathbb{C}}}
\newcommand{\F}{\ensuremath{\mathbb{F}}}

\newcommand{\Fp}{\ensuremath{\mathbb{F}_p}}

\newcommand{\N}{\ensuremath{\mathbb{N}}}
\newcommand{\Q}{\ensuremath{\mathbb{Q}}}

\newcommand{\Z}{\ensuremath{\mathbb{Z}}}

\newcommand{\cB}{\ensuremath{\mathcal{B}}}
\newcommand{\cC}{\ensuremath{\mathcal{C}}}
\newcommand{\cD}{\ensuremath{\mathcal{D}}}

\newcommand{\cH}{\ensuremath{\mathcal{H}}}

\newcommand{\cS}{\ensuremath{\mathcal{S}}}

\newcommand{\cZ}{\ensuremath{\mathcal{Z}}}

\renewcommand{\bar}[1]{\mkern 1.5mu\overline{\mkern-1mu#1\mkern-1mu}\mkern 1.5mu}

\newcommand{\llbracket}{[\mkern-2mu [}
\newcommand{\rrbracket}{]\mkern-2mu ]}

\date{\today}

		\newcommand{\zetapartial}{Z_{N;K}^{c}(s)}
		
		\newcommand{\lan}{\mathcal{L}}
		\newcommand{\struc}{\mathcal{M}}
		\newcommand{\Lan}{\lan_{\mathrm{an}}}
		\newcommand{\ldan}{\lan_{\mathrm{an}}^D}
		\newcommand{\mdan}{\struc^D_{\mathrm{an}}}
		\newcommand{\man}{\struc_{\mathrm{an}}}
		\newcommand{\eqrel}{\mathbin{\mathcal{E}}}
		\newcommand{\tuple}[1]{{\bm{\mathrm{#1}}}}

		\newcommand{\coho}[1]{\mathrm{H}^{#1}}
		\newcommand{\cocy}[1]{\mathrm{Z}^{#1}}
		\newcommand{\cobo}[1]{\mathrm{B}^{#1}}

		\newcommand{\twococyp}{Z_p}
		\newcommand{\twocobop}{B_p}

		\newcommand{\stgroup}{K}
		\newcommand{\uniform}{N}

		\newcommand{\indexPN}{r}
		\newcommand{\indexKpN}{{\indexPN}}
		
		\newcommand{\indexKN}{u}
		
		\newcommand{\indexGN}{{m}}

\DeclareMathOperator{\PIrr}{PIrr} 

\usepackage{babel}
\providecommand{\corollaryname}{Corollary}
\providecommand{\definitionname}{Definition}
\providecommand{\lemmaname}{Lemma}
\providecommand{\propositionname}{Proposition}
\providecommand{\theoremname}{Theorem}

\usepackage[textsize=footnotesize, obeyFinal]{todonotes}

\begin{document}

\title[Rationality of representation zeta functions]{Rationality of representation zeta functions of compact $p$-adic
analytic groups}

\author{Alexander Stasinski and Michele Zordan}
\address{Department of Mathematical Sciences,  Durham University, \allowbreak
	Durham, DH1 3LE, UK}
\email{alexander.stasinski@durham.ac.uk}

\address{Department of Mathematics, Imperial College, London SW7 2AZ, UK.}
\email{mzordan@imperial.ac.uk}
\subjclass[2020]{Primary 20E18; Secondary 20C15, 22E35, 20J06}

\begin{abstract}
We prove that for any FAb compact $p$-adic analytic group $G$, its
representation zeta function is a finite sum of terms $n_{i}^{-s}f_{i}(p^{-s})$,
where $n_{i}$ are natural numbers and $f_{i}(t)\in\Q(t)$ are
rational functions. Meromorphic continuation and rationality of the abscissa of the zeta function follow as corollaries. If $G$ is moreover a pro-$p$ group, we prove
that its representation zeta function is rational in $p^{-s}$. These results
were proved by Jaikin-Zapirain for $p>2$ or for $G$ uniform and
pro-$2$, respectively. We give a new proof which avoids the Kirillov
orbit method and works for all $p$.

\end{abstract}

\maketitle


\section{Introduction}

A group $G$ is called \emph{representation rigid} if the number $r_n(G)$ of isomorphism classes of
(continuous, if $G$ is topological) complex irreducible $n$-dimensional
representations of $G$ is finite, for each $n$. If $G$ is finitely generated profinite and FAb (i.e., $H/[H,H]$
is finite for every finite index subgroup $H$ of $G$), then $G$ is representation rigid (see \cite[Proposition~2]{BassLubotzkyMagidMozes}). For any representation rigid group $G$ we have the formal Dirichlet series
    \[
	Z_G(s)=\sum_{n=1}^{\infty}r_n(G)n^{-s}=\sum_{\rho\in \Irr(G)}\rho(1)^{-s},
    \]
where $\Irr(G)$ denotes the set of irreducible characters of $G$. If the sequence $R_N(G)=\sum_{i=1}^{N}r_i(G)$ grows at most polynomially, $Z_G(s)$ defines a holomorphic function $\zeta_{G}(s)$ on some right half-plane of $\C$, which is called the representation zeta function of~$G$.

Suppose that $G$ is a FAb compact $p$-adic analytic group. In this case, as we will see, the representation zeta function is defined 
and extends to a meromorphic function on $\C$. It is not true in general that $\zeta_{G}(s)$ is a rational function
in $p^{-s}$, but this is not too far from the truth. Any formal Dirichlet series with coefficients in $\Z$ is an element of the ring 
$\Z\llbracket p_1^{-s},p_2^{-s},\dots\rrbracket$, where $p_1,p_2,\dots$ are the primes in $\N$, via 
$\sum_{n}a_n n^{-s} \mapsto \sum_n a_n p_1^{-se_1}p_2^{-se_2}\cdots$, where $p_1^{e_1}p_2^{e_2}\cdots$ is the prime factorisation of $n$.
We say that a Dirichlet series (with integer coefficients) is \emph{virtually rational in
$p^{-s}$} if, as an element of $\Z\llbracket p_1^{-s},p_2^{-s},\dots\rrbracket$, it is of the form
\begin{equation}\label{eqn:virtually-rational}
\sum_{i=1}^{k}n_{i}^{-s}f_{i}(p^{-s}),
\end{equation}
for some natural numbers $k$ and $n_{i}$ and rational functions $f_{i}(t)\in\Q(t)$.

If $Z_G(s)$ defines a zeta function $\zeta_{G}(s)$, we say that $\zeta_{G}(s)$ is virtually rational in $p^{-s}$ if $Z_G(s)$ is, or equivalently, if $\zeta_{G}(s)$ is of the form \eqref{eqn:virtually-rational}, for all $s$ in some right half-plane of $\C$.
In \cite{Jaikin-zeta}, Jaikin-Zapirain proved one of the first and
most fundamental results on representation zeta functions, namely,
that when $p>2$ the zeta function $\zeta_{G}(s)$ is virtually rational in $p^{-s}$ 
and if $G$ is pro-$p$, then $\zeta_{G}(s)$ is rational in $p^{-s}$. 
Moreover, he conjectured that the results hold also for $p=2$, and proved this in the special case when $G$ is uniform pro-$2$. Recall that a pro-$p$ group is called \emph{uniform} (or \emph{uniformly powerful}) if it is finitely generated, torsion free and $[G,G]\leq G^p$ if $p>2$ and $[G,G]\leq G^4$ if $p=2$.
 These results may be compared with analogous (virtual) rationality results proved
earlier by du~Sautoy for subgroup zeta functions (see \cite{duSautoy-rationality}).
Both Jaikin-Zapirain and du~Sautoy rely on a rationality result for
definable integrals, due to Denef and van~den~Dries \cite{Denef-vandenDries}. We will instead use a result of Cluckers (see Section~\ref{subsubsec:Cluckers-rationality}), which differs from that used by Jaikin-Zapirain and du~Sautoy insofar as it, in addition to definable sets, also allows the use of a definable family of equivalence relations.

In \cite{Avni-abscissa}, Avni used Jaikin-Zapirain's virtual rationality theorem as an ingredient to prove that the abscissa of convergence of representation zeta functions of certain arithmetic groups is rational. Jaikin-Zapirain's result has also been fundamental for 
(albeit not always a logical prerequisite of)
work of Larsen and Lubotzky \cite{Larsen-Lubotzky}, Aizenbud and Avni \cite{Aizenbud-Avni-2016}, of Avni, Klopsch, Onn and Voll, e.g., \cite{AKOV-Duke, AKOV-GAFA}, of Budur \cite{bud2021rational}, of Kionke and Klopsch \cite{kioklo2019admissible}, and Zordan \cite{zor2016adjoint,zor2022univariate}.

\subsection{Main result and consequences}
The goal of the present paper is to give a uniform proof of the (virtual) rationality of $\zeta_G(s)$ for $G$ FAb compact $p$-adic analytic and in particular to prove Jaikin-Zapirain's conjecture mentioned above. 
Our main result is:
\begin{thm}
\label{thm:Main}Let $G$ be a FAb compact $p$-adic analytic group.
Then $\zeta_{G}(s)$ is virtually rational in $p^{-s}$. If in addition
$G$ is pro-$p$, then $\zeta_{G}(s)$ is rational in $p^{-s}$.
\end{thm}
This theorem has the following consequences.
\begin{cor}
\label{cor:Main}
	Let $G$ be a FAb compact $p$-adic analytic group. Then the following holds regarding $\zeta_{G}(s)$:
	\begin{enumerate}
		\item it extends meromorphically to the whole complex plane,
		\item it has an abscissa of convergence which is a rational number,
		\item it has no poles at negative integers and $\zeta_{G}(-2)=0$.
	\end{enumerate}
\end{cor}
Here \emph{i)} and \emph{iii)} were previously known consequences of Jaikin-Zapirain's results when $p\neq 2$ or $G$ is uniform pro-$2$, while \emph{ii)} follows from Jaikin-Zapirain's results for all $p$, since for $p=2$ the abscissa of $G$ is the same as for any finite index uniform pro-$2$ subgroup.
In general, \emph{i)} follows immediately from Theorem~\ref{thm:Main}, because any virtually rational function in $p^{-s}$ is clearly meromorphic in all of $\C$. Statement \emph{ii)} follows from the model theoretic rationality result we use (Theorem~\ref{thm:rational_series}), which implies that each rational function appearing in the expression for $\zeta_{G}(s)$ has denominator that is a product of factors of the form $1-p^{i-sj}$, for integers $i,j$ with $j>0$. Moreover, the series $Z_G(s)$ diverges at $s=0$ (since $G$ is an infinite profinite group, hence possesses infinitely many non-equivalent irreducible representations). Thus the abscissa of $\zeta_{G}(s)$ is finite and equals $i/j$, for some $i,j$ as above (note that it does not necessarily equal $\max\{i/j\}$, because some denominators may cancel).
Part \emph{iii)} of Corollary~\ref{cor:Main} was proved in \cite[Theorem~1]{JKGS-zero-at-2}
for all $G$ for which virtual rationality of $\zeta_{G}(s)$ holds. Theorem~\ref{thm:Main} therefore implies that it holds for all $p$.

\subsection{Outline of the paper and the proofs}
In Section~\ref{sec:Basics from model theory} we give the basic definitions and results from model theory that we will use in this paper. 
Similarly, Section~\ref{sec:Prel on proj reps} provides the definitions and results from the theory of projective representations and projective characters, as well as related Clifford theory and group cohomology, that we need in later parts of the paper. 
The remaining sections are devoted to proving our main result. 

Our proof of Theorem~\ref{thm:Main} has the following main features: 
\begin{itemize}
	\item[--] a new argument (i.e., different from the one in \cite[Section~5]{Jaikin-zeta}) for the main part of the proof, namely 
	the rationality of the `partial' zeta series (see below), making systematic use of projective representations and associated cohomology classes; 
	\item[--]
	avoiding the use of Lie algebras and the Kirillov orbit method (which were essential in \cite[Section~5]{Jaikin-zeta}). This is necessary if one hopes to find an analogous proof for $\F_{p}\llbracket t \rrbracket$-analytic groups (see Section~\ref{sec:pos_char}).
\end{itemize}
We now describe the main ideas of the proof in more detail, and point out how it relates to and differs from Jaikin-Zapirain's proof for $p>2$ and the approach in \cite{hrumar2015definable}. The first step is to reduce the (virtual) rationality of $\zeta_{G}(s)$ to the rationality in $p^{-s}$ of the partial zeta series
\[
\zetapartial=\sum_{\theta\in\Irr_{K}^{c}(N)}\theta(1)^{-s},
\]
where $N$ a fixed open normal uniform subgroup of $G$, $K$ is a subgroup of $G$ containing $N$, and  $\Irr_{K}^{c}(N)$ denotes the set of irreducible characters of $N$ with stabiliser $K$ which determine the cohomology class $c$ in the Schur multiplier $\coho{2}(K_p/N)$, where $K_p$ is a pro-$p$ Sylow subgroup of $K$. This reduction step follows \cite[Sections~5-6]{Jaikin-zeta} and uses Clifford theory, together with a result which shows that we can replace $\coho{2}(K/N)$ by $\coho{2}(K_p/N)$ (see Section~\ref{sec:red_partial}).

In Sections~\ref{sec:red_deg_one}-\ref{sec:proof_main} we prove the rationality of the partial 
zeta series and hence Theorem~\ref{thm:Main}. To do this, we show that enumerating characters in 
$\Irr_{K}^{c}(N)$ of given degrees is equivalent to enumerating the classes of a family 
of equivalence relations that is definable with respect 
to an analytic language $\Lan$ of $p$-adic numbers (see Section~\ref{subsubsec:anlang}).
The rationality then follows from a result of Cluckers (see Theorem~\ref{thm:rational_series}). The possibility of using a definable equivalence relation, 
as in \cite{hrumar2015definable},
gives an added flexibility not present in the definability results in \cite{Jaikin-zeta}. We note that in contrast to \cite{hrumar2015definable}, which works with an extended language of rings, we need a $p$-adic analytic language because of the analytic structure of $G$ and du~Sautoy's parametrisation of subgroups via bases, which is one of our key ingredients.

We now describe the contents of Sections~\ref{sec:red_deg_one}-\ref{sec:proof_main} in more detail. In Section~\ref{sec:red_deg_one}, we use some of the theory of projective representations to show that the cohomology class in $\coho{2}(K_p/N)$ associated with a character triple $(K_p,N,\theta)$, for $K_p$ a pro-$p$ \mbox{Sylow} subgroup of $K$, can be obtained from a character triple $(N,N\cap H,\chi)$, where $H$ is an open subgroup of $G$ such that $K_p=HN$, and $\chi$ is of degree one (see Proposition~\ref{prop:Linearisation}). This is a key step, because, just like in \cite{hrumar2015definable}, we can only talk about degree one characters in definability statements. We also introduce a set $X_K$ of pairs $(H,\chi)$, which, modulo a suitable equivalence relation, parametrises the elements in $\Irr(N)$ whose stabiliser is $K$, and a function $\cC:X_K\rightarrow \coho{2}(K_p/N)$ whose fibres parametrise the sets $\Irr_{K}^{c}(N)$, modulo the relation. We then show that these fibres are expressible by a first order formula involving the values of $\chi$, $2$-cocycles and $2$-coboundaries (see Lemma~\ref{lem:first_o_formula_cohomology}). 
The approach in Section~\ref{sec:red_deg_one} is new, compared to \cite{Jaikin-zeta}, and avoids Lie algebra methods by exploiting the monomiality, for projective representations, of $K_p$.

In Section~\ref{sec:proof_main} we use the results of Section~\ref{sec:red_deg_one} to prove that the fibres of $\cC$ and the required equivalence relation are definable in the structure $\man$ of $p$-adic numbers, with respect to the language $\Lan$. 
Among other things, we exploit the known fact about Schur multipliers that every element in $\coho{2}(K_p/N)$ has a cocycle representative of $p$-power order and that we can also choose our coboundaries to have $p$-power order. This implies that we can consider our cocycles and coboundaries as functions with values in $\Q_p/\Z_p$, and hence ultimately as elements of definable sets. 
Once the definability of the fibres of $\cC$, modulo the equivalence relation, is established, an application of Cluckers's rationality result mentioned above finishes the proof of Theorem~\ref{thm:Main}.

In the proof of the definability of the fibres of $\cC$ and the equivalence relation, we adapt some ideas in \cite{hrumar2015definable} to the setting of $p$-adic analytic pro-$p$ groups. The main idea here is that the irreducible characters of $N$ are induced from degree one characters of finite index subgroups, and thus that $\Irr(N)$ can be parametrised (in a many-to-one way) by certain pairs $(H,\chi)$ where $H\leq N$ and $\chi\in\Irr(H)$. Modulo a suitable definable equivalence relation, the parametrisation is bijective and this approach is the reason why the Kirillov orbit method can be avoided. The main new contribution in Section~\ref{sec:proof_main} compared to \cite{hrumar2015definable}, is the definability of the condition for a representation corresponding to a pair $(H,\chi)$ to map to a given $c\in \coho{2}(K_p/N)$ under $\cC$ (note that all of Section~\ref{sec:red_deg_one} is needed for this purpose).

\subsection{Remarks on the positive characteristic case}
\label{sec:pos_char}

It is a natural question to ask whether FAb finitely generated compact $\F_{p}\llbracket t \rrbracket$-analytic groups have virtually rational representation zeta functions. 
This is known for $\mathrm{SL}_2(\F_{q}\llbracket t \rrbracket)$ with $q$ a power of an odd prime (see \cite[Theorem~7.5]{Jaikin-zeta}) and was asked more generally by Larsen and Lubotzky for groups which are $\F_{p}\llbracket t \rrbracket$-points of certain group schemes (see \cite[Problem~6.2]{Larsen-Lubotzky}).

Our proof of Theorem~\ref{thm:Main} 
is a first step towards a solution to this problem, as it avoids the Kirillov orbit method (which is unavailable in characteristic $p$) and Lie algebras (which are often less effective or inadequate in characteristic $p$). 
Moreover, the model theoretic rationality result of Cluckers which we use has a version for uniformly definable equivalence classes over local fields of characteristic $p$, for large enough $p$ (see Nguyen \cite{ngu2016uniform}).
On the other hand, an essential ingredient in our proof of Theorem~\ref{thm:Main} is du~Sautoy's parametrisation of finite index subgroups of $G$, which only works in characteristic $0$. To go further, one will have to narrow down the set of those subgroups of a pro-$p$ group from which all irreducible characters can be obtained by induction of linear characters.
\section{Basics from model theory}
\label{sec:Basics from model theory}
In this section we introduce the necessary tools and notation from model theory.
We refer the interested reader to the first chapters of \cite{mar2006model} and \cite{tenzie2012model} 
for further details. In particular, we refer to 
\cite[Section~1.1]{tenzie2012model} for the concepts of (many-sorted) languages, structures, and definability (with or without parameters). 
By convention, for us a {\em definable set} will always mean an $\emptyset$-definable set (which is also 
called $0$-definable set in \cite[Section~1.1]{tenzie2012model}).\par
In the present paper
we will use languages that model the field $\Q_p$ and 
$p$-adic analytic groups. Preceding authors have used a number of 
languages for these two objects; we review and compare those that are relevant for us.

\subsection{The language of $p$-adic analytic groups}
Let $\N_0$ denote the set of non-negative integers.
It will often be necessary for us to show definability in structures whose underlying set is a pro-$p$ group. 
In these situations, following \cite{duSautoy-rationality}, we will use a language whose constants, 
functions and relations closely resemble the those in a normal pro-$p$ subgroup $N$ inside a 
$p$-adic analytic group $G$.
The following definition is our version of \cite[Definition~1.13]{duSautoy-rationality}.
	\begin{defn}
		\label{def:L_G} Let $N$ be a normal pro-$p$ subgroup of a $p$-adic analytic group $G$.
		The language $\lan_N$ has two sorts $s_1$ (also called the {\em group sort}) and $s_2$, 
		constant symbols in the sort $s_1$ for each element of $N$, and a binary relation 
		symbol $x\mid y$ of sort $(s_1, s_1)$. We have the following function symbols, which all have 
		target sort $s_1$:
			\begin{enumerate}
				\item a binary function symbol $x.y$ of source type $(s_1, s_1)$;
				\item a unary function symbol $x^{-1}$ of source type $s_1$;
				\item a binary function symbol $x^{\lambda}$ of source type $(s_1, s_2)$;
				\item \label{def:L_G:aut} 
				 for each, $g\in G$, a unary function symbol $\varphi_g$ of source type $s_1$.
			\end{enumerate}
	\end{defn}
We define an $\lan_N$-structure $\struc_N$ with underlying set $N$ for the first sort and $\Z_p$ for the second sort. 
The interpretation of the symbols in $\lan_N$ as operations in the group $N$ is immediately suggested by the 
notation, with the exception of $x\mid y$ and the functions $\varphi_g$. The latter is interpreted as the conjugation 
function $N\rightarrow N$, $x\mapsto gxg^{-1}$. In order to interpret $x\mid y$ we will use that $N$ is a pro-$p$ group.
Recall that the lower $p$-series of a pro-$p$ group $H$ is defined as $H_{1}\geq H_{2}\geq\cdots$, where
	\[
	H_{1} = H,\quad\text{and}\quad H_{i+1} = \overline{H_{i}^{p}[H_{i},H]},
	\]
where $\overline{H_{i}^{p}[H_{i},H]}$ denotes the closure of $H_{i}^{p}[H_{i},H]$ as a topological subgroup of $H$. 
We interpret $x\mid y$ as the relation $\omega(x) \geq \omega(y)$, where $\omega$
is as follows.
	\begin{defn}[{\cite[Definition~1.12]{duSautoy-rationality}}] \label{def:omega}
		Define $\omega: N\rightarrow \N \cup \lbrace \infty \rbrace$ by $\omega(g) = n$ if 
		$g\in N_n\setminus N_{n+1}$ and $\omega(1) = \infty$.
	\end{defn}
\subsection{The analytic language of the $p$-adic numbers}\label{subsubsec:anlang} We recall the definition of the language 
used in \cite[Appendix~A]{hrumar2015definable}. 
	\begin{defn}
		The language $\Lan$ is a three-sorted language with a valued field sort VF, a  
		value group sort VG and a residue field sort RF. We have all constants, functions and relations of 
		$\lan_{\mathrm{ring}} = \lbrace +,-, \cdot, 0, 1 \rbrace$ for the valued field and the residue field sort, and all constants, functions 
		and relations of $\lan_{\mathrm{oag}} = \lbrace  +, <, 0 \rbrace$ for the value group sort. In addition, $\Lan$ contains
		\begin{enumerate}
			\item for $m\geq 0$, a function symbol $f$ with source type $\mathrm{VF}^m$ and target sort VF
				for each convergent power series in $m$ variables with coefficients in 
				$\Z_p$;
			\item a function symbol $\mathrm{ord}$ with source type VF and target sort VG;
			\item a function symbol $\overline{\mathrm{ac}}$ with source type VF and target sort RF.
		\end{enumerate}
	\end{defn}
We define an $\Lan$-structure $\man$ 
with underlying sets $\Q_p$ for the valued field sort, $\Z \cup \lbrace -\infty \rbrace$
for the value group, and $\Fp$ (the field with $p$ elements) for the residue field sort. The constants,
functions and relations of $\lan_{\mathrm{ring}}$ and $\lan_{\mathrm{oag}}$ are interpreted in the usual way. The functions $f$ are interpreted as \emph{restricted} analytic functions 
defined by the power series they correspond to: 
	\begin{align*}
        f^{\man}:   \Q_p^m       &\longrightarrow	\Q_p,	&
					\tuple{X}	 &\longmapsto	    {\begin{cases}
														\sum_{\tuple{i} \in \N_0^m} a_{\tuple{i}} X_1^{i_1}\cdots X_m^{i_m}     
														                        &\text{if } \tuple{X} \in \Z_p^m\\
														0				        &\text{otherwise}.
												    \end{cases}}							    
	\end{align*}
The function symbol $\mathrm{ord}$ is interpreted as the valuation map on $\Q_p$ (the valuation of $0$ is $-\infty$). Finally 
the function symbol $\overline{\mathrm{ac}}$ is interpreted as $\overline{\mathrm{ac}}^{\man}: \Q_p \rightarrow \Fp$ 
sending $0$ to $0$ and $x$ to 
	\[
		x p^{- \mathrm{ord}^{\man}(x)} \mod p.
	\]
\par
We show that there is a set of rules thorough which definable sets in $\struc_N$  
may be interpreted as definable sets in $\man$. We shall need the concept of {\em definable interpretation}, 
for which we refer the reader to \cite[Section~5.3]{hod1993model}.\par
%
Let $N$ be a uniform pro-$p$ group and let $n_1,\dots, n_d$ be a minimal set of topological generators for 
$N$. By \cite[Proposition~3.7]{DdSMS}, $N$ is in bijection with $\Z_p^d$ 
via the map $(\lambda_1,\dots, \lambda_d)\mapsto n_1^{\lambda_1}\cdots n_d^{\lambda_d}$. If $g\in N$ is such that 
$ g = n_1^{\lambda_1}\cdots n_d^{\lambda_d}$ for some $\lambda_1, \dots, \lambda_d \in \Z_p$ we say that
$(\lambda_1,\dots, \lambda_d)$ are its $\Z_p$-coordinates (with respect to $n_1, \dots, n_d$).\par

	\begin{lem}
	\label{lem:int_M_N}
		Suppose that $N$ is a uniform normal pro-$p$ group of a compact $p$-adic analytic group $G$. 
		Then $\struc_N$ is definably interpreted in $\man$.
	\end{lem}
		\begin{proof}
		Let $\ldan$ be the analytic language  in 
		\cite[Definition~1.6]{duSautoy-rationality}, \cite[Section~0.6]{Denef-vandenDries} and let $\mdan$
		be the corresponding structure with underlying set $\Z_p$ and the interpretations in 
		\cite[Definition~1.5]{duSautoy-rationality}. 
		Theorem~1.18 and Lemma~1.19 in \cite{duSautoy-rationality} give the definable interpretation of 
		$\struc_N$ in $\mdan$. It is easy to construe the function symbol $D$ of $\ldan$ as a definable function in $\man$, 
		so that the structure $\mdan$ is definably interpreted in $\man$.
		\end{proof}
\subsection{Rationality and definable enumerations in $\man$}\label{subsubsec:Cluckers-rationality}
The following definition is taken from \cite{hrumar2015definable} (Section~6, before Theorem~6.1).

	\begin{defn}
	\label{def:def_family}
		Let $d \in \N$.
		A \emph{definable family} of subsets of $\Q_p^d$ is a definable 
		subset of $\Q_p^d \times \Z$ in $\man$. If $X\subseteq \Q_p^d \times \Z$ 
		is a definable family, we define $X_\ell$ to be the fibre above 
		$\ell \in \Z$ of the projection map $X \to \Z$.
	\end{defn}		
	\begin{defn}
	\label{def:def_family_rel}
		A \emph{definable family of equivalence relations} on a definable family $X$ 
		is an equivalence relation $E$ on $X$ such that if $(x,y) \in E$, then there is an $\ell \in \Z$ such that 
		$x,y \in X_\ell$. This gives a definable equivalence relation on $X_\ell$ for each $\ell \in \Z$, namely $E_\ell = 
		E \cap (X_\ell \times X_\ell)$.
	\end{defn}
	Notice that, since $\N_0$ is a definable subset of $\Z$ in $\man$, we may replace $\Z$ with $\N_0$ in 
	the two definitions above.

	\begin{thm}[{\cite[Theorem~A.2]{hrumar2015definable}}]
	\label{thm:rational_series}
		Let $d \in \N$. Let $E$ be a definable family of equivalence relations in $\man$ on a definable
		family $X\subseteq \Q_p^d \times \N_0$. 
		Suppose that for each $n \in \N_0$ the quotient $X_n/E_n$ is finite, say, of 
		size $a_n$. Then the Poincar\'e series
			\[
				\sum_{n \in \N_0} a_n t^n
			\]
		is a rational power series in $t$ over $\Q$ whose denominator is a product of factors of the 
		form $(1 - p^i t^j)$ for some integers $i,j$ with $j > 0$.
	\end{thm}
		\begin{proof}
			The proof is the same as the one at the end of Appendix~A in \cite{hrumar2015definable}. 
			The only difference is that instead of 
			setting $Y$ to be the set of non-negative integers, we set 
			$Y =  \lbrace n \in \N_0 \mid X_n \neq \emptyset\rbrace$. This set 
			is definable in $\man$ because it is the projection of $X$ on the $\Z$-component of $\Q_p^d \times \Z$. 
			Thus $Y$ is definable because $X$ is. The rest of the proof remains unchanged.
		\end{proof}
\section{Preliminaries on projective representations}
\label{sec:Prel on proj reps}
The main representation theoretic steps of our proof (Section~\ref{sec:red_deg_one})
will use projective representations and projective characters of (pro-)finite
groups. In this section, we collect the definitions and results that
we will need. We use \cite{Isaacs}, \cite{Karpilovsky2} and \cite{Karpilovsky3}
as sources for this theory (precise references for the non-trivial results are given below).

In the following, we regard any $\GL_{n}(\C)$ with its discrete topology.  All the definitions and results in this section apply to finite groups, regarded as discrete profinite groups. In fact, the results are trivial generalisations from the finite groups case because we consider only continuous representations and finite index subgroups. We will however need to apply the results to infinite profinite groups.

From now on, we will consider only continuous representations and their characters. Let $G$ be a profinite group and $N$ an open normal subgroup. We define $\Irr(G)$ to be the set of characters of continuous irreducible complex representations of $G$. 
For any subgroup $K\leq G$ and $\theta\in\Irr(K)$, we denote by $\Irr(G\mid\theta)$ 
the set of irreducible characters of $G$ whose restriction
to $K$ contains $\theta$. The elements of $\Irr(G\mid\theta)$ are said to \emph{lie above} or to \emph{contain} $\theta$. 

For any $K\leq G$, we write 
\[
\Irr_{K}(N)=\{\theta\in\Irr(N)\mid\Stab_{G}(\theta)=K\}
\]
for the irreducible characters of $N$ whose stabiliser under the
conjugation action of $G$ is precisely $K$.

We call $(K,N,\theta)$ a \emph{character triple} if $\theta\in\Irr(N)$ and $K$ fixes $\theta$, that is, if $K\leq\Stab_G(\theta)$. Thus $\Irr_G(N)$ is the set of character triples $(G,N,\theta)$.

A \emph{projective representation} of $G$ is a continuous function
$\rho:G\rightarrow\GL_{n}(\C)$, such that there exists a continuous
function $\alpha:G\times G\rightarrow\C^{\times}$ satisfying
\[
\rho(g)\rho(h)=\rho(gh)\alpha(g,h)\qquad\text{for all }g,h\in G.
\]
The function $\alpha$ is called the \emph{factor set} of $\rho$.
The \emph{projective character} of $\rho$ is the function $G\rightarrow\C$
given by $g\mapsto\tr(\rho(g))$.

 Just like for finite groups, one shows that the factor sets on $G\times G$ are precisely the elements in the group $\cocy{2}(G):=\cocy{2}(G,\C^{\times})$ of continuous $2$-cocycles with values in $\C^{\times}$ (see \cite[(11.6)]{Isaacs}). Moreover, we have the subgroup $\cobo{2}(G):=\cobo{2}(G,\C^{\times})$ of $2$-coboundaries and the cohomology group $\coho{2}(G)=\cocy{2}(G)/\cobo{2}(G)$, the \emph{Schur multiplier} of $G$. It is well known that the Schur multiplier of a finite group is finite (see \cite[(11.15)]{Isaacs}).
 
 Two projective representations $\rho$ and $\sigma$ are said to be \emph{similar} if there exists a $T\in\GL_{n}(\C)$ such that $\rho(g)=T\sigma(g)T^{-1}$, for all $g\in G$. Two projective representations have the same projective character if and only if they are similar. Note that there exists a notion of equivalent projective representations which we will not use.

Projective representations with factor set $\alpha$ naturally correspond to modules for the twisted group algebra $\C[G]^{\alpha}$ (see, e.g., \cite[Section~11]{Isaacs}). It is well known that 
this algebra is semisimple. A projective representation $\Theta$ with factor set $\alpha$ and the character it affords are called \emph{irreducible} if $\Theta$ corresponds to a simple $\C[G]^{\alpha}$-module. We let 
    \[
	\PIrr_{\alpha}(G)
    \]
denote the set of irreducible projective characters of $G$ with factor set $\alpha$.
\begin{defn}
\label{def:stron_ext}
Let $\Theta$ be an irreducible representation of $N$ fixed by $K\leq G$. We say that a projective representation $\Pi$ of $K$ \emph{strongly
extends} (or is a \emph{strong extension} of) $\Theta$ if for all $g\in K$ and $n\in N$, we have:
\begin{enumerate}
\item $\Pi(n)=\Theta(n)$,
\item $\Pi(ng)=\Pi(n)\Pi(g)$,
\item $\Pi(gn)=\Pi(g)\Pi(n)$.
\end{enumerate}
Moreover, in this situation, we say that the projective character
of $\Pi$ \emph{strongly extends} (or is a \emph{strong extension} of) the character of $\Theta$.
\end{defn}
\begin{lem}
\label{lem:factor-set-gn} Let $\Theta$ be an irreducible representation
of $N$ fixed by $K\leq G$ and let $\Pi$ be a projective representation
of $G$ with factor set $\alpha$ such that $\Pi(n)=\Theta(n)$, for
all $n\in N$. Then $\Pi$ strongly extends $\Theta$ if and only
if  $\alpha(g,n)=\alpha(n,g)=1$, for all $g\in G$ and $n\in N$.
\end{lem}
\begin{proof}
The definition of factor set gives that 
	\[
		\Pi(ng)\alpha(n,g)=\Pi(n)\Pi(g),
	\] 
so $\Pi(ng)=\Pi(n)\Pi(g)$ is equivalent to $\alpha(n,g)=1$. Similarly
$\Pi(gn)=\Pi(g)\Pi(n)$ is equivalent to $\alpha(g,n)=1$.
\end{proof}
\begin{thm}
\label{thm:Clifford-map}Let $\Theta$ be an irreducible representation
of $N$ fixed by $K\leq G$. There exists a projective representation $\Pi$
of $K$ which strongly extends $\Theta$. Let $\hat{\alpha}$ be the
factor set of $\Pi$. Then $\hat{\alpha}$ is constant on cosets in
$K/N$, so we have a well-defined element $\alpha\in \cocy{2}(K/N)$
given by
\[
\alpha(gN,hN)=\hat{\alpha}(g,h).
\]
Moreover, we have a well-defined function 
    \[
	\cC_{K}:\{\theta\in\Irr(N)\mid K\leq \Stab_G(\theta)\}\longrightarrow \coho{2}(K/N),\qquad\cC_{K}(\theta)=[\alpha].
    \]
\end{thm}
\begin{proof}
Since $N$ is open in $K$ and every representation of $N$ factors
through a finite quotient, we can reduce to the case of finite groups.
Now, the statements are contained in (the proofs of) \cite[(11.2) and (11.7)]{Isaacs}.
\end{proof}
\begin{lem}
	\label{lem:same_fs}
	Let $\theta$ be an irreducible character of $N$ fixed by $K\leq G$, let 
	$\alpha\in\cocy{2}(K/N)$ be a representative of the cohomology class $\cC_K(\theta)$ and let $\hat{\alpha}$ be the pull-back given by $\hat{\alpha}(g,h)=\alpha(gN,hN)$, for $g,h\in K$.
	Assume that $\alpha$ is trivial on $N \times N$ (i.e, not merely constant but equal to $1$). Then there exists a strong extension of $\theta$ to $K$ with factor set $\hat{\alpha}$.
\end{lem}
\begin{proof}
	Let $\hat{\theta}$ be a strong extension of $\theta$. Let $\hat{\beta}$ be the factor set of $\hat{\theta}$ and 
	$\beta\in\cocy{2}(K/N)$ such that $\beta(gN,hN)=\hat{\beta}(g,h)$.
	By Theorem~\ref{thm:Clifford-map}, there is a $\delta\in \cobo{2}(K/N)$ such that $\alpha = \beta \delta$. 
	Pulling back to $K$, we get $\hat{\alpha}=\hat{\beta}\hat{\delta}$, where $\hat{\delta}(g,h)=\delta(gN,hN)$ for  
	$g,h \in K$. Since $\hat{\delta}$ is a
	coboundary, there is  a $\hat{\gamma}: K \to \C^{\times}$, that is constant on cosets of $N$ and such that 
	$\hat{\delta}(g,h)= \hat{\gamma}(gh)^{-1} \hat{\gamma}(g) \hat{\gamma}(h)$, for  $g,h \in K$.
	As both $\hat{\alpha}$ (by definition) and $\hat{\beta}$ 
	(by Lemma~\ref{lem:factor-set-gn}) are trivial on $N\times N$, the function $\hat{\gamma}\lvert_N$ is 
	a constant homomorphism. We conclude that $\hat{\gamma}\lvert_N = 1$ and $\hat{\gamma} \hat{\theta}$ 
	is a strong extension of $\theta$ with factor set $\alpha$.
\end{proof}

For any $H\leq G$ and factor set $\alpha\in \cocy{2}(G)$,
we denote the restriction of $\alpha$ to $H\times H$ by $\alpha_{H}$.
Suppose that $H$ is open in $G$, and let $\alpha\in \cocy{2}(G)$.
If $\chi$ is a projective character of $H$ with factor set $\alpha_{H}$, we define the \emph{induced projective character}
$\Ind_{H,\alpha}^{G}\chi$ as the character of the induced projective representation given by tensoring 
by the twisted group algebra $\C^{\alpha}[G]$ (see \cite[I, Section~9]{Karpilovsky3}). Then 
$\Ind_{H,\alpha}^{G}\chi$ is a projective character of $G$ with factor set $\alpha$. A projective 
character with trivial factor set is the character of a linear representation and in this case we omit 
the factor set, so that our notation coincides with the standard notation for induced characters of linear representations.

In Section~\ref{sec:red_deg_one} we will freely use
basic facts about projective characters which are direct analogues
of well known results for ordinary characters; for example: 
Frobenius reciprocity \cite[Ch.~1, Lemma~9.18]{Karpilovsky3}, Mackey's
intertwining number formula \cite[Ch.~1, Theorem~8.6]{Karpilovsky3},
and the fact that the inner product $\langle\chi,\chi'\rangle$ of
two projective characters, with $\chi$ irreducible, equals the multiplicity
of $\chi$ as an irreducible constituent of $\chi'$ \cite[Ch.~1, Lemma~8.10]{Karpilovsky3}.
\begin{lem}
\label{lem:projective-monomial}Let $P$ be a pro-$p$ group. Then
any projective representation is induced from a one-dimensional
projective representation of an open subgroup of $P$.
\end{lem}

\begin{proof}
By definition, every projective representation of $P$ factors through
a finite quotient. Since a finite $p$-group is supersolvable, the
result now follows from \cite[Ch.~3, Theorem~11.2]{Karpilovsky2}.
\end{proof}

\subsection{Projective representations and Clifford theory}
If two projective representations of a group $G$ have factor sets $\alpha$ and $\beta$, respectively, then their tensor product has factor set $\alpha \beta$. This is an immediate consequence of the definitions, but is a fact that we will use repeatedly throughout the paper.
The following two lemmas are due to Clifford \cite[Theorems~3-5]{Clifford-1937},
but are not stated in the literature in a form that is useful for us.
\begin{lem}
	\label{lem:Clifford-extensions}
	Let $(K,N,\theta)$ be a character
	triple. Let $\hat{\theta}\in\PIrr_{\hat{\alpha}}(K)$ be a strong
	extension of $\theta$, so that $\cC_{K}(\theta)=[\alpha]$. For any
	$\bar{\pi}\in\PIrr_{\alpha^{-1}}(K/N)$, let $\pi\in\PIrr_{\hat{\alpha}^{-1}}(K)$
	denote the pull-back of $\bar{\pi}$ along the map $K\rightarrow K/N$.
	Then there is a bijection $\PIrr_{\alpha^{-1}}(K/N)\rightarrow \Irr(K\mid\theta)$ given by 
	$\bar{\pi}\mapsto \hat{\theta}\pi$.
\end{lem}

\begin{proof}
	Since $\theta$ factors through a finite group, the statements immediately
	reduce to the case where $K$ and $N$ are finite. The fact that $\bar{\pi}\mapsto\hat{\theta}\pi$
	is a function with the given domain and codomain is proved in \cite[Theorem~5.8\,(ii)]{Nagao-Tsushima} 
	in the context of projective representations. This immediately implies the corresponding fact for projective 
	characters. The fact that it is surjective is \cite[Theorem~5.8\,(i)]{Nagao-Tsushima}.
	We prove injectivity using a simplified version of the argument in
	\cite[p.~545-546]{Clifford-1937}. Let $\Theta$ be a $K$-fixed irreducible
	representation of $N$ and let $\widehat{\Theta}$ be
	a strong extension of $\Theta$ to $K$ with factor set $\hat{\alpha}$.
	Let $\overline{\Pi},\overline{\Pi}'$ be irreducible projective representations
	of $K/N$ with factor set $\alpha^{-1}$, and let $\Pi,\Pi'$ be their
	pull-backs to $K$. Let $d=\dim\widehat{\Theta}=\dim\Theta$, $e=\dim\Pi=\dim\overline{\Pi}$
	and $e'=\dim\Pi'=\dim\overline{\Pi}'$. Assume that $\widehat{\Theta}\otimes\Pi$
	is similar to $\widehat{\Theta}\otimes\Pi'$. Then $\Pi\otimes\widehat{\Theta}$
	is also similar to $\Pi\otimes\widehat{\Theta}'$, that is, there exists
	a $P\in\GL_{de}(\C)$ such that for all $k\in K$, 
	\[
	P^{-1}(\Pi(k)\otimes\widehat{\Theta}(k))P=\Pi'(k)\otimes\widehat{\Theta}(k).
	\]
	Then, for any $n\in N$, we have $P^{-1}(\hat{\alpha}(1,1)^{-1}I_{e}\otimes\Theta(n))P=\hat{\alpha}(1,1)^{-1}I_{e}\otimes\Theta(n)$,
	and thus 
	\[
	P^{-1}(I_{e}\otimes\Theta(n))P=I_{e}\otimes\Theta(n).
	\]
	The matrix $I_{e}\otimes\Theta(n)$ is the value at $n$ of the representation
	$\Theta^{\oplus e}$, so Schur's lemma implies that $P$ is a block matrix consisting
	of $e^2$ scalar blocks of size $d\times d$, that is, $P=Q\otimes I_{d}$, for some
	$Q\in\GL_{e}(\C)$. Hence, for all $k\in K$,
	\[
	0=P^{-1}(\Pi(k)\otimes\widehat{\Theta}(k))P-\Pi'(k)\otimes\widehat{\Theta}(k)=(Q^{-1}\Pi(k)Q-\Pi'(k))\otimes\widehat{\Theta}(k).
	\]
	This implies that $\widehat{\Theta}(k)\otimes(Q^{-1}\Pi(k)Q-\Pi'(k))=0$,
	so since $\widehat{\Theta}(k)$ is non-zero, we must have $Q^{-1}\Pi(k)Q=\Pi'(k)$,
	by the definition of Kronecker product. We have thus proved that if $\widehat{\Theta}\otimes\Pi$
	has the same character as $\widehat{\Theta}\otimes\Pi'$, then $\Pi$ has the same character as $\Pi'$, and this proves the asserted injectivity.
\end{proof}

\begin{lem}
	\label{lem:Clifford-degree-ratios}Let $\theta,\theta'\in\Irr(N)$ be two characters fixed by $K$ such that $\cC_{K}(\theta)=\cC_{K}(\theta')=[\alpha]$, for some $\alpha\in\cocy{2}(K/N)$. Let  $\hat{\theta},\hat{\theta}'\in\PIrr_{\hat{\alpha}}(K)$ be strong extensions of $\theta$ and $\theta'$, respectively, where $\hat{\alpha}$ is the pull-back of $\alpha$ to $K$ (such $\hat{\theta}$ and $\hat{\theta}'$ exist thanks to Lemma~\ref{lem:same_fs}). 
Then there is a bijection  
	$\sigma:\Irr(K\mid\theta)\rightarrow \Irr(K\mid\theta')$,
	 $\hat{\theta}\pi\mapsto\hat{\theta}'\pi$,
	where $\pi$ is the pull-back of $\bar{\pi}\in\PIrr_{\alpha^{-1}}(K/N)$,
   such that
	\[
	\frac{(\hat{\theta}\pi)(1)}{\theta(1)}=\frac{\sigma(\hat{\theta}\pi)(1)}{\theta'(1)}.
	\]
\end{lem}
\begin{proof}
	Lemma~\ref{lem:Clifford-extensions} implies that $\sigma$ is a bijection. For the statement regarding ratios of degrees, it remains to note that 
	    \[
		(\hat{\theta}\pi)(1)=\hat{\theta}(1)\pi(1) \quad\text{and}\quad(\hat{\theta}'\pi)(1)=\hat{\theta}'(1)\pi(1).
	    \]
\end{proof}

The following is a well known result from the cohomology of finite
groups. Note that we write the abelian group structure of cohomology groups multiplicatively as this will be more natural for the cohomology groups we will consider.
\begin{lem}
	\label{lem:basic-group-cohomology}Let $G$ be a finite group of order
	$m$ and let $A$ be a $G$-module. For any integer $i\geq1$, the
	following holds:
	\begin{enumerate}
		\item \label{lem:basic-group-cohomology-1} For any $x\in\coho{i}(G,A)$, we have $x^m=1$. Thus, if $\coho{i}(G,A)$
		is finite and if a prime $p$ divides $|\coho{i}(G,A)|$, then $p$
		divides $m$.
		\item If $P$ is a Sylow $p$-subgroup of $G$, then the restriction homomorphism
		$\res_{G,P}:\coho{i}(G,A)\rightarrow\coho{i}(P,A)$ restricts to an injection
		\[
		\res_{p}:\coho{i}(G,A)_{(p)}\hooklongrightarrow\coho{i}(P,A),
		\]
		where $\coho{i}(G,A)_{(p)}$ is the $p$-torsion subgroup of $\coho{i}(G,A)$.
		Thus, if $\coho{i}(P,A)\allowbreak = 1$ for all Sylow $p$-subgroups and all
		primes $p\mid m$, then $\coho{i}(G,A)=1$. 
	\end{enumerate}
\end{lem}

\begin{proof}
	See, for example, Corollaries~2 and 3 of \cite[Theorem~7.26]{Suzuki}.
\end{proof}
Since any torsion abelian group (not necessarily finite) is a direct
sum of its $p$-torsion subgroups where $p$ runs through all torsion primes (see \cite[Theorem~5.5]{Suzuki}), Lemma~\ref{lem:basic-group-cohomology}\,\ref{lem:basic-group-cohomology-1} 
implies that $\coho{i}(G,A)_{(p)}$ is the $p$-primary component of $\coho{i}(G,A)$. In general, for any torsion abelian group 
$M$ we will denote its $p$-primary component (possibly trivial) by $M_{(p)}$. Similarly, we will write $m_{(q)}$ for the $q$-part of an element $m\in M$. 

\section{Reduction to the partial zeta series}
\label{sec:red_partial}

Let $G$ be a representation rigid profinite group, such that there exists a finite index normal pro-$p$ subgroup $N\leq G$. For example, one can take $G$ to be FAb and compact $p$-adic analytic (see \cite[Corollary~8.34]{DdSMS}). For any $K\leq G$ such that
$N\leq K$, let $K_p$ be a pro-$p$ Sylow subgroup of $K$. Since
$N$ is normal and pro-$p$ we necessarily have $N\leq K_p$. For $c\in \coho{2}(K_p/N)$, define
\[
\Irr_{K}^{c}(N)=\{\theta\in\Irr_{K}(N)\mid\cC_{K_p}(\theta)=c\},
\]
where $\cC_{K_p}$ is the function defined in Theorem~\ref{thm:Clifford-map}.
Note that any two choices of $K_p$ are $G$-conjugate, so up to the natural identification of the groups $\coho{2}(K_p/N)$, for different $K_p$, the set $\Irr_{K}^{c}(N)$ is independent of $K_p$.
We call
\[
\zetapartial=\sum_{\theta\in\Irr_{K}^{c}(N)}\theta(1)^{-s}
\]
a \emph{partial zeta series}. Since $\coho{2}(K_p/N)$ is finite,  there are only finitely
many partial zeta series and 
    \[
	Z_N(s)=\sum_{N\leq K\leq G}\sum_{c\in\coho{2}(K_p/N)} \zetapartial
    \]
for fixed $G$ and $N$.    
Following Jaikin-Zapirain \cite[Section~5]{Jaikin-zeta}, we show how the (virtual)
rationality of $Z_G(s)$, and thus of $\zeta_{G}(s)$, is reduced to the rationality in $p^{-s}$ of the partial zeta series.

Let $(K,N,\theta)$ be a character triple. By Clifford's theorem, $\lambda(1)/\theta(1)$ is an integer for any $\lambda\in\Irr(K\mid\theta)$, so we may define the finite Dirichlet series 
	\[
	f_{(K,N,\theta)}(s)=\sum_{\lambda\in\Irr(K\mid\theta)}\left(\frac{\lambda(1)}{\theta(1)}\right)^{-s}.
	\]

The following result is contained in \cite[Proposition~5.1]{Jaikin-zeta}.
We give a complete proof, which adds several steps involving Schur multipliers.
\begin{lem}
	\label{lem:Jaikins-prop}Let $N$ be a finite index pro-$p$ group
	in $K$ and let $(K,N,\theta)$ and $(K,N,\theta')$ be two character
	triples. If $\cC_{K_p}(\theta)=\cC_{K_p}(\theta')$, then
	\[
		 \cC_{K}(\theta)=\cC_{K}(\theta')\qquad\text{and}\qquad f_{(K,N,\theta)}(s)=f_{(K,N,\theta')}(s).
	\]
\end{lem}

\begin{proof}
	By the remark after Lemma~\ref{lem:basic-group-cohomology}, any $c\in\coho{2}(K/N)$ can be written as $c=\prod_{q}c_{(q)}$, where 
	$q$ runs through the primes dividing $|K/N|$ and $c_{(q)}\in \coho{2}(K/N)_{(q)}$ is the $q$-primary component of $c$. 
    Let $q$ be a prime
	dividing $|K/N|$ and let $K_{q}\leq K$ be such that $K_{q}/N$
	is a Sylow $q$-subgroup of $K/N$ (note that this agrees with our notation $K_p$ for $q=p$).
	By Lemma~\ref{lem:basic-group-cohomology}, $\res_{q}:\coho{2}(K/N)_{(q)}\rightarrow\coho{2}(K_{q}/N)$
	is injective.\par
	Since $\res_{K/N,K_{q}/N}$ has  a $q$-group as target, this homomorphism is trivial on $\coho{2}(K/N)_{(\ell)}$ for primes $\ell\neq q$; hence 
		$\res_{q}(\cC_{K}(\theta)_{(q)}) = \res_{K/N,K_{q}/N}(\cC_{K}(\theta))$.  
	Restricting a strong extension of $\theta$ to $K$ down to $K_p$, it is straightforward to show that 
	$\res_{K/N,K_{q}/N}(\cC_{K}(\theta)) = \cC_{K_{q}}(\theta)$, so
	\begin{equation}
	\res_{q}(\cC_{K}(\theta)_{(q)})=\cC_{K_{q}}(\theta) \quad \text{(and similarly for $\theta'$)}.
	\label{eq:res(C_K(theta))}
	\end{equation}
	
	Now, if $q\neq p$, then $p\nmid|K_{q}/N|$, so by \cite[(8.16)]{Isaacs},
	$\theta$ extends to $K_{q}$, and thus $\cC_{K_{q}}(\theta)=1$.
	By \eqref{eq:res(C_K(theta))} we obtain $\res_{q}(\cC_{K}(\theta)_{(q)})=1$,
	whence $\cC_{K}(\theta)_{(q)}=1$ (by the injectivity of $\res_{q}$). We
	must therefore have $\cC_{K}(\theta)=\cC_{K}(\theta)_{(p)}$, and since
	$\theta$ was arbitrary, we also have $\cC_{K}(\theta')=\cC_{K}(\theta')_{(p)}$.
	For $q = p$, \eqref{eq:res(C_K(theta))} gives 
	\[
	\res_{p}(\cC_{K}(\theta)_{(p)})=\cC_{K_p}(\theta)=\cC_{K_p}(\theta')=\res_{p}(\cC_{K}(\theta')_{(p)}),
	\]
	and we conclude that $\cC_{K}(\theta)_{(p)}=\cC_{K}(\theta')_{(p)}$,
	and thus $\cC_{K}(\theta)=\cC_{K}(\theta')$. 
	Now Lemma~\ref{lem:Clifford-degree-ratios} gives 
	a bijection $\sigma:\Irr(K\mid\theta)\rightarrow\Irr(K\mid\theta')$ such that
	$\lambda(1)/\theta(1)=\sigma(\lambda)(1)/\theta'(1)$. 
	Thus $f_{(K,N,\theta)}(s) = f_{(K,N,\theta')}(s)$ also holds.
\end{proof}

Let $\cS$ denote the set of subgroups $K\leq G$ such that $N\leq K$
and $\Stab_{G}(\theta)=K$, for some $\theta\in\Irr(N)$.
\begin{prop}\label{prop:partial-Main}
	Suppose that $\zetapartial$ is rational in $p^{-s}$, for every $K\in \cS$ and every $c\in \coho{2}(K_p/N)$. Then Theorem~\ref{thm:Main} holds.
\end{prop}
\begin{proof}
By Clifford's theorem, for every $\rho\in\Irr(G)$, there are $|G:\Stab_{G}(\theta)|$
distinct characters $\theta\in\Irr(N)$ such that $\rho\in\Irr(G\mid\theta)$.
Thus

\begin{align*}
Z_{G}(s) & =\sum_{\rho\in\Irr(G)}\rho(1)^{-s}=\sum_{\theta\in\Irr(N)}\frac{1}{|G:\Stab_{G}(\theta)|}\sum_{\rho\in\Irr(G\mid\theta)}\rho(1)^{-s}.
\end{align*}
By standard Clifford theory (see \cite[(6.11)]{Isaacs}), induction yields a bijection between
$\Irr(\Stab_{G}(\theta)\mid\theta)$ and $\Irr(G\mid\theta)$, for
every $\theta\in\Irr(N)$, so 
\[
\sum_{\rho\in\Irr(G\mid\theta)}\rho(1)^{-s}=\sum_{\lambda\in\Irr(\Stab_{G}(\theta)\mid\theta)}(\lambda(1)\cdot|G:\Stab_{G}(\theta)|)^{-s}.
\]
This implies that 
\begin{align*}
Z_{G}(s) & =\sum_{\theta\in\Irr(N)}|G:\Stab_{G}(\theta)|^{-s-1}\sum_{\lambda\in\Irr(\Stab_{G}(\theta)\mid\theta)}\theta(1)^{-s}\left(\frac{\lambda(1)}{\theta(1)}\right)^{-s}\\
 & =\sum_{\theta\in\Irr(N)}|G:\Stab_{G}(\theta)|^{-s-1}\theta(1)^{-s}f_{(\Stab_{G}(\theta),N,\theta)}(s)\\
 & =\sum_{K\in\cS}|G:K|^{-s-1}\sum_{\theta\in\Irr_{K}(N)}\theta(1)^{-s}f_{(K,N,\theta)}(s).
\end{align*}
By Lemma~\ref{lem:Jaikins-prop}, $f_{(K,N,\theta)}(s)=f_{(K,N,\theta')}(s)$,
for $\theta,\theta'\in\Irr(N)$ if $\cC_{K_p}(\theta)=\cC_{K_p}(\theta')$.
By the above, we can therefore write
	\begin{equation*}
		Z_{G}(s)=\sum_{K\in\cS}|G:K|^{-s-1}\sum_{c\in \coho{2}(K_p/N)}f_{K}^{c}(s)\zetapartial
	\end{equation*}
where $f_{K}^{c} (s) :=  f_{(K, N, \theta)}(s)$ for some (equivalently, any) character triple $(K, N, \theta)$ such that $\cC_{K_p}(\theta) = c$.

The set $\cS$ is finite and the group $\coho{2}(K_p/N)$ is also finite by \cite[(11.15)]{Isaacs}.
From the assumption that $\zetapartial$ is rational in $p^{-s}$, it now follows
that $Z_G(s)$, and hence $\zeta_{G}(s)$, is virtually rational. Moreover, if $G$ is pro-$p$, then $|G:K|$ is a power 
of $p$ for any subgroup $K$, and likewise $\lambda(1)$ is a power of $p$ for any $\lambda\in\Irr(K)$, 
so $f_{(K,N,\theta)}(s)$ is a polynomial in $p^{-s}$. Thus, when $G$ is pro-$p$, $Z_G(s)$, and hence 
$\zeta_{G}(s)$, is rational in $p^{-s}$.
\end{proof}

\section{Cohomology classes and degree one characters}
\label{sec:red_deg_one}
To prove the rationality in $p^{-s}$ of the partial zeta series $\zetapartial$ for $G$ FAb compact $p$-adic analytic,
we will prove that the set $\Irr_{K}^{c}(N)$ is in bijection with
the set of equivalence classes of a definable equivalence relation
on a definable set in $\man$. To this end, we need to show that the condition $\cC_{K_p}(\theta)=c$ is equivalent 
to a similar condition where $K_p$ is replaced by a subgroup $H$ of $K_p$ and $\theta$ is replaced by a character 
$\chi$ of $N\cap H$ of degree one. In this section we will state and prove the main technical result allowing for this 
reduction.

As in the previous section, let $G$ be a profinite group possessing a finite index normal pro-$p$ subgroup $N\leq G$. 
All the results in the present section are really theorems about finite groups with trivial generalisations to profinite groups, 
and the reader may assume that $G$ is finite with the discrete topology throughout the section (without changing any of the proofs). 
We work in the profinite setting because this is what we will need to apply the results to in Section~\ref{sec:proof_main}.

For any $K\leq G$ such that $N\leq K$, define the set
\[
\cH(K)=\{H\leq K\mid H\text{ open in }K,\,K=HN\}.
\]
From now on, and until the end of Section~\ref{sec:proof_main}, let $N\leq K\leq G$ be fixed.

\begin{lem}
\label{lem:IndRes}Let $\gamma\in \cocy{2}(K_p)$, $H\in\cH(K_p)$
and $\eta\in\PIrr_{\gamma_{H}}(H)$ be of degree one.
If $\Ind_{N\cap H, \gamma_N}^{N}\Res_{N\cap H}^{H}\eta$ or $\Res_{N}^{K_p}\Ind_{H,\gamma}^{K_p}\eta$
is irreducible, then
\[
\Ind_{N\cap H,\gamma_N}^{N}\Res_{N\cap H}^{H}\eta=\Res_{N}^{K_p}\Ind_{H,\gamma}^{K_p}\eta.
\]
\end{lem}
\begin{proof}
By Mackey's induction-restriction formula and Frobenius reciprocity for projective representations,  
\begin{multline*}
	\left\langle \Ind_{N\cap H, \gamma_N}^{N}\Res_{N\cap H}^{H}\eta,\Res_{N}^{K_p}\Ind_{H,\gamma}^{K_p}\eta\right\rangle 
						=\left\langle \Res_{N\cap H}^{H}\eta,\Res_{N\cap H}^{K_p}\Ind_{H,\gamma}^{K_p}\eta\right\rangle \\
	=\sum_{\bar{g}\in(N\cap H)\backslash K_p/H}\left\langle \Res_{N\cap H}^{H}\eta,
						\Ind_{N\cap H\cap\leftexp{g}{H}, \gamma_{N\cap H}}^{N\cap H}
							\Res_{N\cap H\cap\leftexp{g}{H}}^{\leftexp{g}{H}}\leftexp{g}{\eta}\right\rangle \\
	\geq\left\langle \Res_{N\cap H}^{H}\eta,\Res_{N\cap H}^{H}\eta\right\rangle =1.
\end{multline*}
Here $g\in K_p$ denotes an arbitrary representative of $\bar{g}$. Since $K_p=HN$ and 
	\[
		|K_p:N|\cdot|N:N\cap H|=|K_p:H|\cdot|H:N\cap H|=|K_p:H|\cdot |HN:N|,
	\]
we deduce that $|N:N\cap H| = |K_p:H|$. Hence $\Ind_{N\cap H, \gamma_N}^{N}\Res_{N\cap H}^{H}\eta$
and $\Res_{N}^{K_p}\Ind_{H,\gamma}^{K_p}\eta$ have the same degree, so if one of them is irreducible, they are equal.
\end{proof}

For $H\leq K_p$ such that $K_p=HN$, we let $\widetilde{f}_{H}:\cocy{2}(H/(N\cap H))\rightarrow \cocy{2}(K_p/N)$
be the isomorphism induced by pulling back cocycles along the isomorphism
$K_p/N\rightarrow H/(N\cap H)$. We describe this isomorphism more explicitly. 
Since $K_p = HN$, every coset in $K_p/N$ contains a unique coset in $H/(N\cap H)$. 
Then, for $\alpha\in \cocy{2}(H/(N\cap H))$ and $g,g' \in K_p$, 
	\begin{equation}\label{f-tilde-explicit}
		\widetilde{f}_{H}(\alpha)(gN, g'N) = \alpha(h(N \cap H), h'(N \cap H))
	\end{equation}
where $h,h'$ are such that $h(N \cap H)\subseteq gN$ and $h'(N \cap H)\subseteq g'N$. 
Moreover, for $\beta\in \cocy{2}(K_p/N)$ and $h,h'\in H$, 
	\[
		\widetilde{f}_{H}^{-1}(\beta)(h(N \cap H), h'(N \cap H)) = \beta(hN, h'N).
	\]
We denote by $f_{H}$ the corresponding induced isomorphism 
	\[
		\coho{2}(H/(N\cap H))\longrightarrow \coho{2}(K_p/N).
	\]
\begin{prop}
\label{prop:Linearisation}
	Let $(K,N,\theta)$ be a character triple. Then there exists an $H\in\cH(K_p)$
	and a character triple $(H,N\cap H,\chi)$ such that:
	\begin{enumerate}
		\item \label{enu:i} $\chi$ is of degree one,
		\item \label{enu:ii} $\theta=\Ind_{N\cap H}^{N}\chi$,
		\item \label{enu:iii} $\cC_{K_p}(\theta)=f_{H}(\cC_{H}(\chi))$.
	\end{enumerate}
	Moreover, let $H\in\cH(K_p)$ be such that $(H,N\cap H,\chi)$ is a
	character triple with $\chi$ of degree one, such that $(K,N,\theta)$
	is a character triple, where $\theta=\Ind_{N\cap H}^{N}\chi$. Then
	$\cC_{K_p}(\theta)=f_{H}(\cC_{H}(\chi))$. 
\end{prop}
\begin{proof}
Assume that $(K,N,\theta)$ is a character triple. By Theorem~\ref{thm:Clifford-map},
there exists an $\alpha\in \cocy{2}(K_p/N)$ such that $[\alpha]=\cC_{K_p}(\theta)$
and a $\hat{\theta}\in\PIrr_{\hat{\alpha}}(K_p)$ strongly extending
$\theta$. 
Note that by Lemma~\ref{lem:factor-set-gn}, $\hat{\alpha}(n,x)=\hat{\alpha}(x,n)=1$, for all $n\in N$ and $x\in K_p$,
so in particular, $\hat{\alpha}_{N}=1$. 

By Lemma~\ref{lem:projective-monomial}, there exist an
open subgroup $H$ of $K_p$ and $\eta\in\PIrr_{\hat{\alpha}_{H}}(H)$
of degree one such that $\hat{\theta}=\Ind_{H, \hat{\alpha}}^{K_p}\eta$. Then $\theta=\Res_{N}^{K_p}\Ind_{H, \hat{\alpha}}^{K_p}\eta$
is irreducible, so 
\begin{align*}
1 & =\left\langle\Res_{N}^{K_p}\Ind_{H, \hat{\alpha}}^{K_p}\eta,\Res_{N}^{K_p}\Ind_{H, \hat{\alpha}}^{K_p}\eta\right\rangle=\\
 & =\sum_{\bar{g}\in N\backslash K_p/H}\sum_{\bar{h}\in N\backslash K_p/H}\left\langle \Ind_{N\cap\leftexp{g}{H}}^{N}\Res_{N\cap\leftexp{g}{H}}^{\leftexp{g}{H}}\leftexp{g}{\eta},\Ind_{N\cap\leftexp{h}{H}}^{N}\Res_{N\cap\leftexp{h}{H}}^{\leftexp{h}{H}}\leftexp{h}{\eta}\right\rangle \\
 & =\sum_{\bar{g}\in K_p/HN}\sum_{\bar{h}\in K_p/HN}\left\langle \Ind_{N\cap\leftexp{g}{H}}^{N}\Res_{N\cap\leftexp{g}{H}}^{\leftexp{g}{H}}\leftexp{g}{\eta},\Ind_{N\cap\leftexp{h}{H}}^{N}\Res_{N\cap\leftexp{h}{H}}^{\leftexp{h}{H}}\leftexp{h}{\eta}\right\rangle \\
 & \geq\sum_{\bar{g}\in K_p/HN}\left\langle \Ind_{N\cap\leftexp{g}{H}}^{N}\Res_{N\cap\leftexp{g}{H}}^{\leftexp{g}{H}}\leftexp{g}{\eta},\Ind_{N\cap\leftexp{g}{H}}^{N}\Res_{N\cap\leftexp{g}{H}}^{\leftexp{g}{H}}\leftexp{g}{\eta}\right\rangle \\
 & \geq |K_p:HN|.
\end{align*}
Thus, $|K_p:HN|=1$, and so $K_p=HN$, that is, $H\in\cH(K_p)$.

Next, let $\chi=\Res_{N\cap H}^{H}\eta$. Then $\chi$ is fixed by
$H$, and Lemma~\ref{lem:IndRes} (with $\gamma=\hat{\alpha}$) implies that
$\theta=\Ind_{N\cap H}^{N}\chi$. 
Moreover, let $\alpha_{H}\in \cocy{2}(H/(N\cap H))$ be defined as 
	\[
		\alpha_{H}(h(N\cap H),h'(N\cap H))=\hat{\alpha}_{H}(h,h') \quad \text{ for } h,h' \in H.
	\]
Then $f_{H}([\alpha_{H}])=\cC_{K_p}(\theta)$. We conclude that $\cC_{K_p}(\theta)=f_{H}(\cC_{H}(\chi))$,
because $\eta$ strongly extends $\chi$.\par
Assume now that $(H,N\cap H,\chi)$ and $(K,N,\theta)$ are as in the second part of the statement.
By Theorem~\ref{thm:Clifford-map},
there exists a $\beta\in \cocy{2}(H/(N\cap H))$ and a $\hat{\chi}\in\PIrr_{\hat{\beta}}(H)$
strongly extending $\chi$, such that $[\beta]=\cC_{H}(\chi)$.  
Let $\gamma\in \cocy{2}(K_p)$ be the pull-back of $\widetilde{f}_{H}(\beta)\in \cocy{2}(K_p/N)$.
Then, for any $h,h'\in H$,  
    \[
	\gamma_{H}(h,h')=\gamma(h,h')=\widetilde{f}_H(\beta)(hN,h'N)=\beta(h(N\cap H),h'(N\cap H))=\hat{\beta}(h,h'),
    \]
where in the second to last step we have used \eqref{f-tilde-explicit}.
Thus $\gamma_H=\hat{\beta}$, and since $\theta$ is irreducible, Lemma~\ref{lem:IndRes} (with $\eta=\hat{\chi})$ implies that 
\[
\theta=\Ind_{N\cap H, \gamma_N}^{N}\Res_{N\cap H}^{H}\hat{\chi}=\Res_{N}^{K_p}\Ind_{H, \gamma}^{K_p}\hat{\chi}.
\]
Hence $\Ind_{H, \gamma}^{K_p}\hat{\chi}$ is an extension of $\theta$ and we show that it is in fact a strong extension (see Definition~\ref{def:stron_ext}).
Indeed, as $\gamma$ is constant on cosets of $N$ in $K_p$, 
\[
\gamma(x,n)=\gamma(hn',n)=\gamma(h,1)=\gamma_{H}(h,1)=\hat{\beta}(h,1)=1,
\]
where we have written $x=hn'$, with $h\in H$, $n'\in N$ and $\hat{\beta}(h,1)=1$
by Lemma~\ref{lem:factor-set-gn}, because $\hat{\beta}$ is the
factor set of a strong extension. In a similar way, we show that $\gamma(n, x) = 1$; thus, by Lemma~\ref{lem:factor-set-gn},
we conclude that $\Ind_{H, \gamma}^{K_p}\hat{\chi}$ strongly extends $\theta$. 
Since $\Ind_{H, \gamma}^{K_p}\hat{\chi}$ has factor set 
$\gamma$, which descends (modulo $N$) to $\widetilde{f}_H(\beta)$,
it follows that $\cC_{K_p}(\theta)=[\widetilde{f}_{H}(\beta)]=f_{H}([\beta])=f_{H}(\cC_H(\chi))$.
\end{proof}
It will be useful for us to state a consequence of Proposition~\ref{prop:Linearisation} in terms of a 
commutative diagram. To this end, let $X_K$ 
\label{def:X_K}
be the set of pairs 
$(H,\chi)$ with $H\in\cH(K_p)$, where:
\begin{enumerate}
\item $(H,N\cap H,\chi)$ is a character triple.
\item $\chi$ is of degree one,
\item $\Ind_{N\cap H}^{N}\chi\in \Irr_K(N)$. 
\end{enumerate}
Note that $\theta \in\Irr_K(N)$ means that $K=\Stab_G(\theta)$, and not merely that $K$ is contained in the stabiliser. 
Define the function
\[
\cC:X_K\longrightarrow \coho{2}(K_p/N)
\]
by $\cC(H,\chi)=f_{H}(\cC_{H}(\chi))$. 
\begin{cor}
\label{cor:surj_coho}
The function $X_K\rightarrow\Irr_{K}(N)$, $(H,\chi)\mapsto\Ind_{N\cap H}^{N}\chi$ is surjective and the following 
diagram commutes:
    \[
	\begin{tikzcd}[column sep=0.4cm] 
	    X_K \arrow{r} \arrow{dr}[swap]{\cC} & \Irr_{K}(N)\arrow{d}{\cC_{K_p}}\\
	    {}				  & \coho{2}(K_p/N).
	\end{tikzcd}
    \]
\end{cor}
	\begin{proof}
	Every $\theta\in\Irr_K(N)$ defines a character triple $(K,N,\theta)$. Thus, the surjectivity follows from the first statement 
	in Proposition~\ref{prop:Linearisation}. The commutativity of the diagram follows by the second statement in 
	Proposition~\ref{prop:Linearisation}.
	\end{proof}
\section{Rationality of the partial zeta series}
\label{sec:proof_main}
From now on, let $G$ be a FAb compact $p$-adic analytic group and let $N\leq G$ be a normal uniform subgroup. As in Section~\ref{sec:red_deg_one}, let $K\leq G$ be such that $N\leq K$ and fix a pro-$p$ Sylow subgroup $K_p$ of $K$.
In this section we show that the set of characters $\Irr_K^c(N)$, for each $c\in\coho{2}(K_p/N)$, 
is in bijection with a set of equivalence classes under a definable equivalence relation in $\man$. 
We deduce from this that each partial zeta series is rational in $p^{-s}$ and hence prove 
Theorem~\ref{thm:Main}.
\subsection{Bases for $p$-adic analytic groups}
\label{subsec:Good-bases}

Recall from Section~\ref{sec:red_deg_one} that $\cH(K_p) \allowbreak = \{H\leq K_p\mid H\text{ open in }K_p,\,K_p=HN\}$. 
In this section, we describe du~Sautoy's parametrisation of $\cH(K_p)$.

One starts by parametrising open subgroups of $N$. The following definition is from 
\cite[p.~259]{duSautoy-Segal-in-Horizons} and is equivalent to \cite[Definition~2.2]{duSautoy-rationality}. 
Some properties characterising open subgroups of $N$ and 
some notation are necessary to state it. A subgroup $H$ of $N$ is open if and only if it contains $N_{m}$
for some $m\geq 1$, where $N_m$ denotes the $m$-th term of the 
lower $p$-series of $N$. Moreover, as $N$ is uniform, raising to the power of $p$ induces an 
isomorphism $N_i/ N_{i + 1}\rightarrow N_{i + 1}/ N_{i + 2}$ and $N_{i+1}$ is the Frattini subgroup 
of $N_i$, for all $i\in \N$ (see \cite[Lemma~2.4, Definition~4.1\,(iii)]{DdSMS}). 
Thus $N_i/N_{i+1}$ is an $\F_p$-vector space, and denoting by $d = \dim_{\mathbb{F}_p} N/N_2$ 
the minimum number of topological generators for $N$, each quotient $N_i/N_{i+1}$ is isomorphic to 
$\mathbb{F}_p^d$. Recall the function $\omega$ in 
Definition~\ref{def:omega}.
\begin{defn}
\label{def:good_basis}
	Let $H\leq N$ be open with $N_{m}\leq H$. A $d$-tuple $(h_{1},\dots,h_{d})$
	of elements in $H$ is called a \emph{good basis} for $H$ if
	\begin{enumerate}
		\item 
		$\omega(h_{i})\leq\omega(h_{j})$ whenever $i\leq j$, and
		\item 
		\label{def:good_basis_2}
		for each $n\leq m$, the set
		\[
		\left\{ h_{i}^{p^{n-\omega(h_{j})}}N_{n+1}\bigm|i\in\{1,\dots,d\},\,\omega(h_{i})\leq n\right\} 
		\]
		is a basis for the $\F_{p}$-vector space $(N_{n}\cap H)N_{n+1}/N_{n+1}$.
	\end{enumerate}
\end{defn}
Notice that the definition is constructive so a good basis for an open subgroup of $N$ always exists. 
Notice also that a good basis for $N$ is just an ordered minimal set of topological generators of $N$ and that, by 
\cite[Lemma~2.4\,(i)]{duSautoy-rationality}, if $H$ is an open subgroup of $N$ and
$(h_1, \dots, h_d)$ is a good basis for $H$, then for every $h \in H$ there are 
$\lambda_1, \dots, \lambda_d \in \Z_p$ such that 
	\[
		h = h_1^{\lambda_1}\cdots h_d^{\lambda_d}.
	\]
The recursive construction in the proof of \cite[Lemma~2.4\,(i)]{duSautoy-rationality} 
implies that $\lambda_1, \dots, \lambda_d$ are unique with the property above.
\begin{rem}
Good bases give a many-to-one parametrisation of the set of finite 
index subgroups of $N$ in terms of $p$-adic analytic coordinates. Indeed the set of 
good bases is definable in $\struc_N$ by \cite[Lemma~2.8]{duSautoy-rationality}. 
By Lemma~\ref{lem:int_M_N}, using $\Z_p$-coordinates for $N$, the set of good bases 
 is interpreted as a definable set in $\man$.
 \end{rem}
The parametrisation of $\cH(K_p)$ is obtained by extending the parametrisation given 
by good bases. Let $\indexPN=|K_p:N|$. Fix a left transversal $(y_{1},\dots,y_{\indexPN})$ for 
$N$ in $K_p$ with $y_1=1$.
Every coset $y_iN$ contains a unique coset $x(N\cap H)$, with $x\in H$. Thus, $x=y_it_i$ for some 
$t_i\in N$, and we conclude that there exist elements $t_1,\dots,t_{\indexPN}\in N$ such that 
$(y_1t_1,\dots,y_{\indexPN}t_{\indexPN})$ is a left transversal for $N\cap H$ in $H$. 
The following definition is from \cite[Definition~2.10]{duSautoy-rationality};
see also \cite[p.~261]{duSautoy-Segal-in-Horizons} (note that we use left cosets instead of du~Sautoy's right coset convention).
\begin{defn}
\label{def:basis}
	Let $H\in\cH(K_p)$. A $(d+\indexPN)$-tuple $(h_{1},\dots,h_{d},t_{1},\dots,t_{\indexPN})$
	of elements in $N$ is called a \emph{basis} for $H$ if
	\begin{enumerate}
		\item $(h_{1},\dots,h_{d})$ is a good basis for $N\cap H$, and
		\item $(y_{1}t_{1},\dots,y_{\indexPN}t_{\indexKpN})$ is a (left) transversal for $N\cap H$ 
		in $H$.
	\end{enumerate}
\end{defn}
If $(h_{1},\dots,h_{d},t_{1},\dots,t_{\indexPN})$
is a basis for $H\in  \cH(K_p)$, it follows from the definition that
	\[
		H = \bar{\langle h_1, \dots h_d, y_1t_1, \dots,  y_\indexPN t_\indexPN\rangle}.
	\]
In particular, unlike a good basis for $N$, a basis for $H$ need not be a (topological) generating set for $H$.
Notice moreover that a basis of $H$ always exists: it suffices to construct a good basis 
$(h_{1},\dots,h_{d})$ of $N\cap H$ 
as described in Definition~\ref{def:good_basis} and then find $t_1,\dots, t_\indexPN$ using that
each coset of $N$ in $K_p$ contains a unique coset of $N\cap H$ in 
$H$ because $K_p = HN$. 
The groups, transversals and bases appearing above are illustrated
by the following diagrams:
\[
\begin{tikzcd}[column sep={3em,between origins},row sep={2.5em,between origins}] 
{} &  K_p\arrow[dash]{ddl}\arrow[dash]{dr} & {}\\
{} & {} & H\arrow[dash]{ddl}\\
N\arrow[dash]{dr} & {} & {}\\ 
{}  & N\cap H & {}
\end{tikzcd}
\qquad\qquad\qquad
\begin{tikzcd}[column sep={3.5em,between origins},row sep={3em,between origins}] 
{} &  (y_1,\dots,y_\indexPN)\arrow[dash]{ddl}\arrow[dash]{dr} & {}\\
{} & {} & (y_1t_1,\dots, y_\indexPN t_\indexPN)\arrow[dash]{ddl}\\
(t_1,\dots,t_\indexPN)\arrow[dash]{dr} & {} & {}\\ 
{}  & (h_1,\dots,h_d) & {}
\end{tikzcd}
\]
\begin{rem}
By \cite[Lemma~2.12]{duSautoy-rationality}, the set of bases is definable in $\struc_N$, hence, by 
Lemma~\ref{lem:int_M_N}, can be interpreted as a definable set in $\man$ by passing to $\Z_p$-coordinates for~$N$.
\end{rem}
\subsection{The fibres of $\cC$ in terms of degree one characters.}
\label{subsec:Def_of_fibres}
From now on, let $c \in \coho{2}(K_p/\uniform)$. The aim of this section is to show that the set
$\cC^{-1}(c)$ may be characterised by a predicate involving only elements of $N$ and degree one 
characters of finite index subgroups of $N$. We will at the end of the section produce an $\Lan$ formula 
for the fibre of $\cC$. We therefore start by reducing the range for $c$ to a cohomology group with 
values in the group of roots of unity of order a power of $p$. 
In order to do this, we need to set up some notation. Let $W\leq \C^{\times}$ be the group of roots of 
unity. This is a torsion abelian group so it splits as
\begin{equation*}
W = \prod_{\ell \text{ prime}} W_{(\ell)}
\end{equation*}
where $W_{(\ell)}\leq W$ is the group of roots of unity of order a power of $\ell$. It is clear that $W$ 
is a divisible group so by \cite[XX, Lemma~4.2]{Lang-Algebra} it is injective in the category of abelian groups, 
hence it is complemented in $\C^{\times}$. We may therefore fix a homomorphism $\C^{\times}\to W$, and for each 
prime $\ell$ denote by $\pi_\ell:\C^\times \to W_{(\ell)}$ the homomorphism obtained by composing with the 
projection $W\to W_{(\ell)}$.

If $f$ is a function with image inside $\C^{\times}$ and $\ell$ is a prime, we define 
    \[
	f_{(\ell)} = \pi_{\ell}\circ f.
    \]
Note that if $f$ has finite order, 
that is, if $f$ has image in 
$W$, then $f_{(\ell)}$ coincides with the $\ell$-primary 
component of $f$. Moreover, since $\pi_{\ell}$ is a homomorphism, 
$(ff')_{(\ell)} = f_{(\ell)}f'_{(\ell)}$ for all $f,f'$ with the same domain and with codomain $\C^{\times}$.\par

We introduce the following groups:
	\begin{align*}
		\twococyp     &= \cocy{2}(K_p/ \uniform, W_{(p)}),	&
		&\text{and}&
		\twocobop     &= \cobo{2}(K_p/ \uniform, W_{(p)}).								    
	\end{align*}
By \cite[(11.15)]{Isaacs} and its proof, every class in $\coho{2}(K_p/ \uniform)$ has a representative in 
$\twococyp$. Moreover, let $\delta \in \cobo{2}(K_p/ \uniform) \cap \twococyp$. Then, by definition, there is a function 
$\varphi: K/N \to \C^{\times}$ such that, for all $a,b \in K/N$, 
	\[
		\delta(a,b) = \varphi(a) \varphi(b) \varphi(ab)^{-1}.
	\]
Now $\delta$ has values in $W_{(p)}$ already,  so, for all $a,b \in K/N$, 
	\[
		\delta (a,b) = \delta_{(p)} (a,b)= \varphi_{(p)} (a) \varphi_{(p)} (b) \varphi_{(p)} (ab)^{-1}.
	\]
Thus $\delta \in \twocobop$, and $\cobo{2}(K_p/ \uniform) \cap \twococyp = \twocobop$. 
It follows that the inclusion of $\twococyp$ in $\cocy{2}(K_p/ \uniform)$ induces an 
isomorphism $\coho{2}(K_p/ \uniform)\cong \twococyp/\twocobop$.\par
We now turn to describing the fibres of the map $\cC$.   
Define $a_{ij}\in N$ and 
$\gamma:\lbrace 1, \dots, \indexPN \rbrace^{2}\rightarrow\lbrace 1,\dots,\indexPN \rbrace$ by
\begin{equation}
\label{eq:gamma}
y_{i}y_{j}=y_{\gamma(i,j)}a_{ij}.
\end{equation}
We also define the inner automorphisms $\varphi_{i}=\varphi_{y_i}:G\rightarrow G$,
$\varphi_{i}(g)=y_{i}gy_{i}^{-1}$, for $g\in G$. 
The purpose of the following lemma is to show that the fibres of $\cC$ are given by a first order statement involving only values of 
degree one characters, cocycles and coboundaries. 
\begin{lem}
\label{lem:first_o_formula_cohomology}
Let $(H,\chi)\in X_K$ and 
$t_1,\dots,t_{\indexPN}\in N$ such that $(y_1t_1,\dots,y_{\indexPN}t_{\indexPN})$ is a left transversal 
for $N\cap H$ in $H$. Let $\alpha\in \twococyp$ such that $[\alpha] = c$. 
Then $\cC(H,\chi)= c$ if and only if there exists $\delta\in \twocobop$
such that for all $n,n'\in N\cap H$ and all $i,j\in\{1,\dots,\indexPN\}$,
we have
\begin{equation}
\label{eq:formula_cohomology}
\chi(t_{\gamma(i,j)}^{-1}a_{ij}\varphi_{j}^{-1}(t_{i}n)t_{j}n')\alpha(y_{i}N, y_{j}N)\delta(y_{i}N, y_{j}N)=\chi(nn').
\end{equation}
\end{lem}
	\begin{proof}
	We have $\cC(H,\chi)=[\alpha]$ if and only if there exists a strong
	extension $\hat{\chi}\in\PIrr_{\hat{\beta}}(H)$ of $\chi$, with
	$\hat{\beta}\in \cocy{2}(H)$ such that $f_{H}([\beta])=[\alpha]$.
	Since every two strong extensions of $\chi$ to $H$ define the same 
	element 
		\[
			\cC_{H}(\chi)\in \coho{2}(H/(N\cap H)),
		\] 
	we may
	without loss of generality assume that $\hat{\chi}$ is given by
	\begin{equation}
	\label{eq:def_ext}
	\hat{\chi}(y_{i}t_{i}n)=\chi(n),
	\end{equation}
	for all $n\in N\cap H$ and $y_{i}t_{i}$. Thus $\cC(H,\chi)=[\alpha]$
	if and only if there exists $\beta\in \cocy{2}(H/(N\cap H))$
	such that $f_{H}([\beta])=[\alpha]$ and such that for all $n,n'\in N\cap H$
	and all $i,j\in\{1,\dots,\indexPN\}$, 
		\[
		\hat{\chi}(y_{i}t_{i}n y_{j}t_{j}n')\hat{\beta}(y_{i} t_{i} n,y_{j}t_{j}n')=\hat{\chi}(y_{i}t_{i}n)\hat{\chi}(y_{j}t_{j}n').
		\]
	Notice that, by definition, $\hat{\chi}$ has values in $W_{(p)}$. Thus we may strengthen the last equivalence by assuming
	that $\hat{\beta} \in \cocy{2}(H, W_{(p)})$ and consequently $\beta \in \twococyp$.
	The last equation is equivalent to
		\[
		\hat{\chi}(y_{i}t_{i}n y_{j}t_{j}n')\beta(y_{i}t_{i}(N\cap H), y_{j}t_{j}(N\cap H))=\chi(nn').
		\]
	Furthermore, $y_it_i(N\cap H)\subseteq y_iN$, so  $f_{H}([\beta])=[\alpha]$ 
	if and only if there exists a $\delta\in \twocobop$ such that for all $i,j\in\{1,\dots,\indexPN\}$, 
		\[
		\beta(y_{i}t_{i}(N\cap H), y_{j}t_{j}(N\cap H))=\alpha(y_{i}N,y_{j}N)\delta(y_{i}N,y_{j}N).
		\]
	Notice that here we were able to restrict the range for $\delta$ to $\twocobop$, because we could assume 
	that $\beta \in \twococyp$ and we chose $\alpha \in \twococyp$.
	
	Combining these two statements of equivalence we obtain that $\cC(H,\chi)=[\alpha]$
	if and only if there exists $\delta\in \twocobop$ such
	that for all $n,n'\in N\cap H$ and for all $i,j\in\{1,\dots,\indexPN\}$, 
		\[
		\hat{\chi}(nt_{i}y_{i}n't_{j}y_{j})\alpha(y_{i}N,y_{j}N)\delta(y_{i}N,y_{j}N) = \chi(nn').
		\]
	Hence, to finish the proof, we need to show that
		\[
		\hat{\chi}(nt_{i}y_{i}n't_{j}y_{j})
		= \chi(t_{\gamma(i,j)}^{-1}a_{ij}\varphi_{j}^{-1}(t_{i}n)t_{j}n').
		\]
	Indeed, this follows from \eqref{eq:def_ext} and the identities 
    \begin{align*}	  
y_{i}t_{i}ny_{j}t_{j}n' & =y_{i}y_{j}y_{j}^{-1}t_{i}ny_{j}t_{j}n
  =y_{i}y_{j}\varphi_{j}^{-1}(t_{i}n)t_{j}n'\\
 & =y_{\gamma(i,j)}a_{ij}\varphi_{j}^{-1}(t_{i}n)t_{j}n' =y_{\gamma(i,j)}t_{\gamma(i,j)}t_{\gamma(i,j)}^{-1}a_{ij}\varphi_{j}^{-1}(t_{i}n)t_{j}n',
   \end{align*}
noting that $t_{\gamma(i,j)}^{-1}a_{ij}\varphi_{j}^{-1}(t_{i}n)t_{j}n'$ 
lies in $H$ (since $y_{i}t_{i}ny_{j}t_{j}n'$ and $y_{\gamma(i,j)}t_{\gamma(i,j)}$ do), and therefore in $N\cap H$. 
   \end{proof}
\subsection{Definable sets for $\twococyp$ and $\twocobop$.}
We will now introduce the definable sets that will be used to interpret predicates 
quantifying over $\twococyp$ and $\twocobop$.
\begin{rem}
\label{ass:prufer}
It is well-known that $\Q_p / \Z_p$ is isomorphic 
to   $W_{(p)}$ via the map $\iota: a/p^m+\Z_p \mapsto e^{2\pi i a/p^m}$ 
(cf.~\cite[Lemma~8.7]{hrumar2015definable}). 
\end{rem}
	\begin{lem}
	\label{lem:schur_def_sets}
	Define $\mathcal{Z}$ and $\mathcal{B}$ to be the sets of matrices $(z_{ij}) \in {\M_{\indexPN}(\Q_p)}$ such that the map
	    \[
	        ( y_iN, y_jN)
							\longmapsto \iota(z_{ij} + \Z_p), \quad \text{for}\  {i,j\in \lbrace 1, \dots, \indexPN \rbrace}
		\]
	is in $\twococyp$ and $\twocobop$ respectively. Then $\mathcal{Z}$ and $\mathcal{B}$ are definable subsets of $\Q_p^{r^2}$ in 
	$\man$.
	\end{lem}
	\begin{proof} 
	Let $(z_{ij}) \in \M_{\indexPN}(\Q_p)$ and let $\alpha$ be the the map 
	$K_p/\uniform \times K_p/\uniform \to \Q_p/\Z_p$ defined as
	    \[
	        \alpha( y_iN, y_jN)
							\longmapsto \iota( z_{ij} + \Z_p), \qquad i,j \in \lbrace 1, \dots, \indexPN \rbrace.
	    \] 
	Imposing that $\alpha$ satisfy the $2$-cocycle identity, we obtain that $(z_{ij}) \in \mathcal{Z}$ 
	if and only if for all $i,j,k \in \lbrace 1, \dots, \indexPN \rbrace$, 
		\[
			z_{\gamma(i,j)\, k} + z_{ij} = z_{i \, \gamma(j,k)} +  z_{jk} \mod \Z_p,
		\]
	where  $\gamma$ is as defined in \eqref{eq:gamma}.
	Notice that $\Z_p$ is definable in $\man$, 
	hence equivalence modulo $\Z_p$ is a definable relation. It follows that the set $\mathcal{Z}$ 
	is definable in $\man$.\par
	The set $\mathcal{B}$  is also definable in $\man$. Indeed, $\delta \in  \twocobop$ if and only if 
		\[
    			 \delta(x,y) = \varphi(x) \varphi(y)  \varphi(xy)^{-1},
		\]
	for some function $\varphi: K_p/\uniform \rightarrow \Q_p/\Z_p$. 
	We parametrise the functions 
	$K_p/N\rightarrow \Q_p$ by the $\indexPN$-tuples 
	of their values on $y_1N,\dots, y_{\indexPN}N$. In these coordinates, 
	we obtain that $\alpha \in \twocobop$ if and only if there are $b_1, \dots, \allowbreak b_\indexPN \in \Q_p$ 
	with the property that for all $i,j \in \lbrace 1, \dots, \indexPN \rbrace$,
		\[
			z_{ij} = b_i + b_j - b_{\gamma(i,j)}      \mod	\Z_p.
		\]
	This is a definable predicate in $\man$ so $\mathcal{B}$  is definable in $\man$.
	\end{proof}
\subsection{Definability of the fibres of $\cC$}
We now find a definable parametrisation of the fibres of $\cC$ in Corollary~\ref{cor:surj_coho}.
We need the following lemma to definably express $K$-stability of characters by an $\Lan$-formula.
	\begin{lem}
	\label{lem:conj_induced}
		Let $M$ be a finite index subgroup of $N$ and $\chi$ be a character of $M$ of degree one. 
		Then, for all $g \in G$,
			\[
				\leftexp{g}{\big(\mathrm{Ind}_M^N \chi\big)} = \mathrm{Ind}_{\leftexp{g}{M}}^N\leftexp{g}{\chi}.
			\]
		Moreover if $M'$ is another finite index subgroup of $N$ and $\chi, \chi'$ are degree one characters 
		of $M$ and $M'$ respectively, such that $\Ind_M^N \chi$ and $\Ind_{M'}^N \chi'$ are irreducible, then 
		$\Ind_M^N \chi = \Ind_{M'}^N \chi'$ if and only if there exists $g \in N$ such that 
		$\Res^{\leftexp{g}{M}}_{\leftexp{g}{M} \cap M'} \leftexp{g}{\chi} = \Res^{M'}_{\leftexp{g}{M} \cap M'} \chi'$.
	\end{lem}
		\begin{proof}
		The proof of the first statement is a routine check using the formula for an induced character.
		The second statement follows from Mackey's theorem 
		(cf.\ \cite[Proposition~8.6\,(c)]{hrumar2015definable}).
		\end{proof}
We are ready to construct the definable set parametrising $\cC^{-1}(c)\subseteq X_K$. Let from now on $n_1, \dots, n_d \in N$ be a minimal 
topological generating set for $N$.

     \begin{prop}
	\label{pro:X_definable}
	Let $c \in \coho{2}(K_p/ \uniform)$ and let $\cD^c$	 be the set of pairs 
	$(\tuple{\lambda}, \tuple{\xi})$, $\tuple{\lambda}\in \M_{d\times (d + \indexPN)}(\Z_p)$, $\tuple{\xi}=(\xi_1,\dots,\xi_d)\in \Q_p^{d}$  such that:
		\begin{enumerate}
			\item \label{pro:X_definable_1} the columns of $\tuple{\lambda}$ are the $\Z_p$-coordinates with respect to 
			$n_1, \dots, n_d$ of a basis $(h_1,\dots, h_d, t_1,\dots, t_\indexPN)$ for some subgroup 
			$H \in \cH(K_p)$.
			\item \label{pro:X_definable_2} The function $\{h_1,\dots, h_d\} \rightarrow \Q_p/\Z_p$, $h_i\mapsto \xi_i + \Z_p$, 
			extends to a (necessarily unique) continuous $H$-invariant homomorphism 
				\[
					\chi: N \cap H\longrightarrow \Q_p/\Z_p.
				\]
			\item \label{pro:X_definable_3} $\Ind_{N\cap H}^N (\iota \circ \chi) \in 
			\Irr_\stgroup(\uniform)$,
			\item \label{pro:X_definable_4} $\cC(H, (\iota \circ \chi)) = c$.
		\end{enumerate}
	Then $\cD^c$ is a definable subset of $\Q_p^{d(d+r+1)}$
	 in $\man$.
	\end{prop}
		\begin{proof}
		Condition \ref{pro:X_definable_1} is expressible by an $\Lan$-formula by 
		\cite[Lemma~2.12]{duSautoy-rationality}. Following  the proof of 
		\cite[Lemma~8.8]{hrumar2015definable}, we show that if  \ref{pro:X_definable_1}
		holds, then \ref{pro:X_definable_2} holds if and only if:
\renewcommand{\theenumi}{\emph{\alph{enumi})}}  
		\begin{enumerate}
			\item \label{pro:X_equivalent_1}
				there exists $(\mu_{ij}) \in \M_{d}(\Z_p)$ 
				whose columns are the $\Z_p$-coordinates of a good basis for some finite index 
				normal subgroup $M$ of $N \cap H$;
			\item \label{pro:X_equivalent_3}
				there exist $\xi \in \Q_p$, $r_1, \dots, r_d \in \Z_p$, and $h \in N\cap H$ such 
			that the order of $\xi$ in $\Q_p/\Z_p$ is $\lvert N\cap H : M\rvert$ and for every 
			$i,j \in \lbrace 1,\dots, d\rbrace$,
				\begin{align*}
					 h_j		    &=\leftexp{t_i^{-1}}{\varphi_i^{-1}(h^{r_j})} \mod M	&
					 & \text{and}&
					r_i \xi		 &= \xi_i \mod \Z_p.
				\end{align*}
		\end{enumerate}\par
\renewcommand{\theenumi}{\emph{\roman{enumi})}}
		Suppose that conditions \ref{pro:X_definable_1} and \ref{pro:X_definable_2} in the statement hold. Then 
		$\chi$ factors through a finite quotient of $N \cap H$. Set $M = \Ker \chi$ and choose
		$(\mu_{ij}) \in \M_{d}(\Z_p)$ such that its 
		columns are the $\Z_p$-coordinates of a good basis of $M$.
		Condition \ref{pro:X_equivalent_1} is immediately satisfied. 
		Moreover the group $(N\cap H)/M$ is cyclic, because it is isomorphic to a subgroup of $\C^{\times}$.
		This, together with the $H$-invariance of $\chi$, implies condition \ref{pro:X_equivalent_3} for $h \in N\cap H$ 
		such that $(N\cap H)/M = \langle h M \rangle$, $\xi \coloneqq \chi (h)$ and $r_1, \dots, r_d \in \Z$ such that, for 
		$i \in \lbrace 1, \dots, d\rbrace$, $h_i M = h^{r_i} M$.\par
		Conversely, assume there are  $(\mu_{ij})\in \M_{d}(\Z_p)$, $h \in H$, and $\xi \in \Q_p$
		such that \ref{pro:X_equivalent_1} and \ref{pro:X_equivalent_3}
		hold. We define a continuous homomorphism $\chi: N \cap H\rightarrow \Q_p/\Z_p$ as follows.
		By \cite[Lemma~1.19]{duSautoy-rationality} the map $\Z_p\rightarrow N\cap H$ defined by 
		$\lambda \mapsto h^{\lambda}$ is analytic in the $\Z_p$-coordinates of $N$ 
		and therefore it is continuous. Since $M$ is an open subgroup, we may find a neighbourhood 
		of $U$ of $0$ such that $h^\lambda \in M$ for all $\lambda \in U$. Now, $\Z$ is dense in $\Z_p$, 
		so, for all $i \in \lbrace 1,\dots, d\rbrace$, we may find $s_i \in (r_i + U) \cap \Z$. Clearly, since 
		$s_i \in r_i + U$, we have 
			\[
				 h^{s_i} M = h^{r_i} M = h_i M,
			\]
		showing that $(N\cap H) / M$ is cyclic with generator $h M$.\par
		By assumption, the order of $\xi + \Z_p$ in $\Q_p/\Z_p$ is equal to the order of $h M$ in $(N \cap H)/M$, thus 
		there is an injective homomorphism 
			\[
				\beta:  (N \cap H)/M\rightarrow \Q_p/\Z_p \text{ defined by } h M \mapsto \xi + \Z_p.
			\] 
		We define $\chi: N \cap H\rightarrow \Q_p/\Z_p$ to be the composition 
		of $\beta$ with the canonical map $N \cap H \rightarrow (N \cap H)/M$. The latter is continuous 
		by \cite[Proposition~1.2]{DdSMS}, so $\chi$ is a continuous homomorphism. Since $y_1 = 1$ by assumption, 
		$t_1 \in N \cap H$. So, for all $j \in \lbrace 1, \dots, d\rbrace$, 
			\[
				\chi(h_j) = \chi( \leftexp{t_1^{-1}}{h^{r_j}}) = r_j \xi + \Z_p= \xi_j + \Z_p.
			\] 
		Similarly, for $i, j \in \lbrace 1, \dots, d\rbrace$, we have 
		$\chi(\leftexp{t_i^{-1}}{\varphi_i^{-1}(h_j)}) = \xi_j + \Z_p$ showing that $\chi$ is $H$-invariant.\par
		Conditions \ref{pro:X_equivalent_1} and \ref{pro:X_equivalent_3} become $\Lan$-formulas
		by passing to $\Z_p$-coordinates with respect to $n_1, \dots, n_d$ and via the interpretation 
		of $\struc_N$ in $\man$ of Lemma~\ref{lem:int_M_N}. Notice that membership in $N \cap H$ 
		can be expressed by means of the $\Z_p$-coordinates of $N$ because we assumed that 
		$h_1, \dots, h_d$ is a good basis by \ref{pro:X_definable_1}. Moreover, equivalence modulo 
		$M$ is definable in $\man$, as we have a good basis for $M$. Finally, the condition on the order 
		of $\xi$ is equivalent to 
			\[
				\left(h^{(\xi^{-1})} \in M\right) \wedge 
						\left(\forall\, \eta \in \Q_p \, (\ord(\eta) > \ord (\xi)
						\Rightarrow  h^{(\eta^{-1})} \notin M)\right).
			\]\par
		We now show that condition \ref{pro:X_definable_3} is definable. 
		To simplify notation we will throughout the rest of the proof identify the group $\Q_p/ \Z_p$ with $W_{(p)}$ through $\iota$. 
		Under this identification, we re-define $\chi = \iota \circ \chi$.\par
		First we show that the irreducibility 
		of $\mathrm{Ind}_{N\cap H}^N \chi$ is expressible as an $\Lan$-formula. Indeed, by Mackey's 
		irreducibility criterion $\Ind_{N\cap H}^N \chi$ is irreducible if and only if 
			\[
				 \forall\, g\in N\ : \Big(\big(\forall\, h\in N\cap H,\,\chi(\leftexp{g}{h})
										= \chi(h)\big)\Longrightarrow g\in H\Big).
			\]
		By \ref{pro:X_definable_1} and \ref{pro:X_definable_2} we may rewrite the formula above in terms of 
		$\tuple{\lambda}$, $\tuple{\xi}$ and of the $\Z_p$-coordinates in $N$. By Lemma~\ref{lem:int_M_N}, this gives 
		an $\Lan$-formula for the irreducibility statement in condition \ref{pro:X_definable_3}. 
		To conclude the proof that this condition gives rise to a definable set, we show that $K$-invariance can also 
		be expressed by an $\Lan$-formula. 
		Indeed, let $\indexKN = \lvert K : N \rvert$ and 
		$\indexGN = \lvert G : N \rvert$. Fix $y_{\indexKpN + 1}, \dots, y_{\indexGN} \in G$ such that 
		$(y_1, \dots, y_\indexKN)$ and $(y_1, \dots, y_\indexGN)$ are left transversals of $N$ in $K$ and 
		$G$ respectively. Recall that, for $g\in G$, we denote by $\varphi_g$ the 
		conjugation by $g$ on $N$. Let 
			\begin{align*}
				C_K     & = \lbrace \varphi_{y_i} \mid i \in \lbrace 1, \dots, \indexKN\rbrace \rbrace, &
				C_G     &= \lbrace \varphi_{y_i} \mid i \in \lbrace 1, \dots, \indexGN \rbrace \rbrace.
			\end{align*}
		Notice that $C_K \subseteq C_G$. By Lemma~\ref{lem:conj_induced}, 
		the stabiliser of  $\Ind_{N\cap H}^{N}\chi$ is equal to $K$ if and only the following statement holds:
		\begin{equation}
		\label{eq:stab_is_K}
			\forall\, \varphi \in C_G\ : \Big(\Ind_{N\cap H}^{N}\chi=\Ind_{\varphi (N\cap H)}^{N}\chi \circ\varphi^{-1} 
			\Longleftrightarrow \varphi \in C_K \Big).
		\end{equation}
		Fix
		$i\in\lbrace 1,\dots, \indexGN \rbrace$. Lemma~\ref{lem:conj_induced} with 
		$M = N \cap H$, $M' = \leftexp{y_i}{(N \cap H})$ and $\chi' = \leftexp{y_i}{\chi}$ implies that 
		$\Ind_{N\cap H}^{N}\chi=\Ind_{\varphi_{y_i}(N\cap H)}^{N}\chi \circ\varphi_{y_i}^{-1}$ if and only if
			\[
				\exists\, g \in N,\ \forall\, h\in N\cap H\ : \big(\leftexp{g}{h} \in \leftexp{y_i}{(N \cap H}) \Longrightarrow \chi(h) 
						= \leftexp{y_i}{\chi}(\leftexp{g}{h})\big).
			\]
		Again, by  \ref{pro:X_definable_1} and \ref{pro:X_definable_2}, we may write the latter in terms of
		$\tuple{\lambda}$, $\tuple{\xi}$ and of the $\Z_p$-coordinates in $N$. Substituting in \eqref{eq:stab_is_K} finishes the 
		proof that 
		condition \ref{pro:X_definable_3}
		is definable. Notice that we are allowed to conjugate elements of 
		$N$ by elements of $G$ because there are corresponding function symbols $\varphi_g$ in $\lan_N$ (and these
		are interpreted as definable functions in $\man$ by Lemma~\ref{lem:int_M_N}).\par
		Finally we show that also \ref{pro:X_definable_4} can be expressed by 
		an $\Lan$-formula. Fix $\beta \in \twococyp$ such that $[\beta] = c$. By Lemma~\ref{lem:first_o_formula_cohomology}, 
		condition \ref{pro:X_definable_4} is equivalent to
			\begin{multline}
			\label{eq:coho_is_c}
				\exists\, \delta \in \twocobop\ :
				\Bigg(\bigwedge_{i,j\in\{1,\dots,\indexPN\}}
					\forall\, n, n' \in N \cap H\\
						\Big(
	       \chi(t_{\gamma(i,j)}^{-1}a_{ij}\varphi_{j}^{-1}(t_{i}n)t_{j}n') \beta(y_{i}N, y_{j}N) \delta(y_{i}N, y_{j}N)= \chi(nn').
						\Big)\Bigg).
			\end{multline}
		We describe how the above 		
		is translated to an $\Lan$-formula. First, the multiplications by $a_{ij}, t_i, t_j$ etc.\ 
		are analytic functions $N \to N$, which have corresponding expressions in $\Lan$. Secondly, 
		we need to choose a $(b_{ij}) \in \cZ$ corresponding to $\beta$ and replace $\beta(y_{i}N,y_{j}N)$ by $b_{ij}$. In general, 
		this could cause the final $\Lan$-formula to have parameters in $\Q_p$, which we might not be able to eliminate. 
		However, every class in $\Q_p/\Z_p$ has a representative in $\Q$. If a rational representative is chosen for 
		each $\beta(y_{i}N,y_{j}N)$ and we multiply every equation in the conjunction by a suitable power of $p$, then we may assume the $b_{ij}$'s are  
		parameters in $\Z$, which (in the final $\Lan$-formula) we can replace with $\Lan$-expressions evaluating to them in $\man$ 
		(e.g.\ $b\in \Z_{\geq 0}$ is replaced by $1 + \cdots + 1$ where $1$ appears $b$ times).\par	
		Next, we replace 
		$(\exists \delta \in \twocobop)$ by $(\exists  (d_{ij}) \in \cB)$ where $\cB$ is as in Lemma~\ref{lem:schur_def_sets},
		and we replace $\delta(y_{i}N,y_{j}N)$ by $d_{ij}$. 
		Finally, by \ref{pro:X_definable_1} and \ref{pro:X_definable_2} again, 
		we write the equations in \eqref{eq:coho_is_c} as congruences $\bmod\, \Z_p$ with variables
		$\tuple{\lambda}$, $\tuple{\xi}$, the $\Z_p$-coordinates of $n$ and $n'$, and $d_{ij}$. This gives the 
		required $\Lan$-formula. 
		\end{proof}
Proposition~\ref{pro:X_definable} shows that 
there is a surjective map $\Psi: \cD^c \rightarrow \cC^{-1}(c)$ defined by 
$(\tuple{\lambda}, \tuple{\xi})\mapsto (H, \chi)$ where $H\in\cH(K_p)$ is the subgroup corresponding to 
the basis $(h_1,\dots, h_d, t_1,\dots, t_\indexPN)$ of Proposition~\ref{pro:X_definable}~\ref{pro:X_definable_1}
and $\chi$ is as in Proposition~\ref{pro:X_definable}~\ref{pro:X_definable_2}.
\subsection{Finishing the proof of Theorem~\ref{thm:Main}}
We write the partial zeta series as a generating function enumerating the equivalence classes of a family 
of definable equivalence relations. We conclude rationality of the partial zeta series
by Theorem~\ref{thm:rational_series}. Theorem~\ref{thm:Main} then follows from 
Proposition~\ref{prop:partial-Main}.\par
We start by constructing a definable equivalence 
relation on $\cD^c$ whose equivalence classes will be in bijection with 
$\Irr_K^{c}(N)$. 
Let $(\tuple{\lambda}, \tuple{\xi}), (\tuple{\lambda}', \tuple{\xi}') \in \cD^c$ and
let $(H, \chi) = \Psi (\tuple{\lambda}, \tuple{\xi})$ and $(H', \chi') = \Psi (\tuple{\lambda}', \tuple{\xi}')$. We define an 
equivalence relation $\eqrel$ on $\cD^c$ by 
	\[
		((\tuple{\lambda}, \tuple{\xi}), (\tuple{\lambda}', \tuple{\xi}')) \in \eqrel  \Longleftrightarrow 
									\mathrm{Ind}_{N\cap H}^N \chi = \mathrm{Ind}_{N\cap H'}^N \chi'.
	\]
	\begin{lem}
	The relation $\eqrel$ is definable in $\man$.
	\end{lem}
		\begin{proof}
		Let $(H, \chi), (H', \chi')$ be as above. Lemma~\ref{lem:conj_induced} 
		implies that $\Ind_{N\cap H}^N \chi = \Ind_{N\cap H'}^N \chi'$ if and only if 
			\[
				\exists\, g \in N,\ \forall\, h\in N \cap H\ \left(\leftexp{g}{h} \in N \cap H' \Longrightarrow \chi(h) = \chi'(\leftexp{g}{h})\right).
			\]
		Writing this in the $\Z_p$-coordinates of $N$ we obtain an $\Lan$-formula, which, after restricting  to 
		the definable set $\cD^c$, gives the $\Lan$-formula defining 
		$\eqrel$.
		\end{proof}
Composing $\Psi$  with the surjective map $X_K\rightarrow \Irr_K(N)$ of Corollary~\ref{cor:surj_coho}
induces a bijection between the set of equivalence classes $\cD^c/\eqrel$ and  $\Irr_K^c(N)$. We now use 
this bijection to produce a definable family of equivalence relations giving the partial zeta series. 
For $(\tuple{\lambda}, \tuple{\xi}) \in \cD^c$,
write $(h_1(\tuple{\lambda}), \dots, h_d(\tuple{\lambda}))$ for the good basis associated 
with $\tuple{\lambda}$ by Proposition~\ref{pro:X_definable}~\ref{pro:X_definable_1}. The function $f:\cD^c\rightarrow \N_0$ given by

	\[
	 (\tuple{\lambda}, \tuple{\xi}) \longmapsto \sum_{i = 1}^{d} {\omega(h_i(\tuple{\lambda})) - 1}
	\]
is definable  in $\man$ because $\struc_N$ is definably interpreted in $\man$ and, 
under this interpretation, $\omega$ becomes a definable function 
by \cite[Theorem~1.18\,(iv)]{duSautoy-rationality}. Notice that, if $\Psi(\tuple{\lambda}, \tuple{\xi}) = (H, \chi)$, 
then, by the discussion preceding \cite[Lemma~2.8]{duSautoy-rationality}, $p^{f(\tuple{\lambda}, \tuple{\xi})}$ 
is the index of $N\cap H$ in $N$, which is equal to the degree of $\Ind_{N\cap H}^{N} \chi$.\par 
We now define a family of equivalence relations that will allow us to apply Theorem~\ref{thm:rational_series} to  $\zetapartial$. Our approach 
is the same as in \cite{hrumar2015definable} after Remark~6.3 but for a fixed prime $p$. We set 
	$
		D^c 
		\subseteq \Q_p^{d(d+r+1)} \times \N_0
	$ 
to be the graph of $f$ and 
define an equivalence relation $E \subseteq D^c \times D^c$ by 
	\[
		((x, n), (x',n')) \in E^c \iff (x,x') \in \eqrel.
	\] 
Notice that, unlike \cite{hrumar2015definable}, we do not require $n = n'$ here, as this condition is vacuous whenever $(x,x') \in \eqrel$.
Clearly $D^c$ is a definable family of subsets of $\Q_p^{d(d+r+1)}$ and $E^c$ is a definable family of 
equivalence relations on $D^c$.\par
For all $n \in \N_0$ the fibre of $f$ above $n$ 
is a union of $\eqrel$-equivalence classes. 
Therefore
the set $D^c_n/E^c_n$ is in bijection with the subset of characters of degree $p^n$ in $\Irr_K^c(N)$.
It follows that
	\[
		\zetapartial = \sum_{n \in \N_0} \#(D^c_n/E^c_n) p^{-ns}.
	\] 
Applying Theorem~\ref{thm:rational_series} to the series above
we deduce that $\zetapartial$ is a rational function in $p^{-s}$. This concludes the proof.
\begin{acknowledgement*}
We thank Benjamin Martin for helpful suggestions and Andrei Jaikin-Zapirain for reading a preliminary version of 
this paper and giving valuable comments. We also thank Marcus du Sautoy for answering our questions about bases, 
Gabriel Navarro for his comments on Proposition~\ref{prop:Linearisation}, Raf Cluckers for answering our questions 
about \cite[Theorem~A.2]{hrumar2015definable}, and Benjamin Klopsch for helpful discussions.

The second author was financially supported by Project G.0792.18N of the Research Foundation - Flanders (FWO), 
by the Hausdorff Research Institute for Mathematics (Universit\"at Bonn), by the University of Auckland, and by Imperial College London. 
Part of the work on this paper was funded by a Durham University Travel Grant and LMS Scheme 4 grant 41678.
\end{acknowledgement*}
\bibliographystyle{abbrv}
\def\cprime{$'$}

\end{document}